\documentclass[11pt,a4paper]{article}
\usepackage[utf8]{inputenc}
\usepackage[OT1,T1]{fontenc}
\usepackage{amsmath,amssymb,mathrsfs,stmaryrd,bm}
\usepackage[colorlinks,citecolor=black,linkcolor=black,urlcolor=black,
  pdfpagemode=None,pdfstartview=FitH]{hyperref}
\usepackage[all]{xy}
\SelectTips{cm}{10}

\newtheorem{theorem}{Theorem}[section]
\newtheorem{lemma}[theorem]{Lemma}
\newtheorem{proposition}[theorem]{Proposition}
\newtheorem{corollary}[theorem]{Corollary}
\newenvironment{proof}{\trivlist
  \item[\hskip\labelsep{\itshape Proof.}]\upshape}{\nobreak\noindent
  $\square$\endtrivlist}
\newenvironment{other}[1]{\refstepcounter{theorem}\trivlist
  \item[\hskip\labelsep{\itshape #1~\thesection.\arabic{theorem}.}]
  \upshape}{\endtrivlist\bigbreak}

\DeclareMathOperator\Hom{Hom}

\DeclareMathOperator\gr{gr}
\DeclareMathOperator\height{ht}
\DeclareMathOperator\Pol{Pol}
\DeclareMathOperator\Irr{Irr}
\DeclareMathOperator\rep{rep}
\newcommand\id{\mathord{\mathrm{id}}}
\newcommand\Gr{\mathord{\mathcal G\kern-.4ptr}}
\newcommand\dimvec{\underline{\mathop{\mathrm{dim}}}}
\newcommand\inter[2]{\llbracket#1,#2\rrbracket}
\newcommand\polyline[1]{[\kern1pt#1\kern1pt]}
\newcommand\BM{{\text{BM}}}

\begin{document}
\title{On Mirković--Vilonen polytopes}
\author{Pierre Baumann}
\date{}
\maketitle

\begin{abstract}
\noindent
Mirković--Vilonen polytopes encode in a geometrical way the numerical data
present in the Kashiwara crystal $B(\infty)$ of a semisimple group $G$.
We retrieve these polytopes from the coproduct of the Hopf algebra
$\mathscr O(N)$ of regular functions on a maximal unipotent subgroup
$N$ of $G$. We bring attention to a remarkable behavior that the
classical bases (dual canonical, dual semicanonical, Mirković--Vilonen)
of $\mathscr O(N)$ manifest with respect to the extremal points of these
polytopes, which extends the crystal operations. This study leans on a
notion of stability for graded bialgebras.
\end{abstract}

\section{Introduction}
\label{se:Introduction}
Let $G$ be a split semisimple, or more generally a split reductive group
over a field $k$ of characteristic zero. The representation theory of
$G$ is well understood thanks to Weyl's theory of the highest weight.
Let $B$ be a Borel subgroup of $G$, let $T$ be a maximal torus of $B$,
let $P$ be the character lattice of $T$, and let $N$ be the unipotent
radical of $B$. Then all the irreducible rational $G$-modules can be
embedded as $N$-modules in the algebra $\mathscr O(N)$ of regular
functions on $N$, which thus plays a central role in the theory.
Moreover, the images of these embeddings are invariant under the action
of $T$ on $\mathscr O(N)$ deduced from the action on $N$ by conjugation.

It is therefore natural to pay attention to the $N$-submodules of
$\mathscr O(N)$ that are invariant under the action of $T$. A basic
piece of combinatorial information associated to such modules is the set
of their weights. Conversely, given a convex polytope $K$ in the vector
space $P\otimes_{\mathbb Z}\mathbb R$, one can consider the largest
$T$-invariant $N$-submodule $\mathscr S(K)$ of $\mathscr O(N)$ whose
weights lie in $K$. The map $K\mapsto\mathscr S(K)$ then defines
a filtration of $\mathscr O(N)$ indexed by the set of all convex
polytopes, endowed with the partial order given by inclusion.
The places in the filtration where the jumps occur,
to wit the convex polytopes $K$ such that
$\mathscr S(K)\neq\sum_{K'\subsetneq K}\mathscr S(K')$, certainly
reflect the structure of the $N$-module $\mathscr O(N)$.

It turns out that these polytopes are already known in the literature:
they are the Mirković--Vilonen polytopes introduced by
Anderson~\cite{Anderson}. The name comes from the Geometric Satake
Equivalence, where the irreducible rational $G$-modules are realized as
the intersection cohomology of affine Schubert varieties. In this context,
Mirković and Vilonen selected algebraic cycles whose fundamental classes
form a basis of the cohomology~\cite{MirkovicVilonen}, and the polytopes
appear as the images of these cycles by the moment map of a natural
Hamiltonian action of a torus.

Another occurrence of these polytopes was found in
\cite{BaumannKamnitzerTingley}. The context here is the representation
theory of the preprojective algebra $\Lambda$ build from the Dynkin
diagram of $G$ (assumed to be of simply-laced type).
To a $\Lambda$-module $M$, one can associate its Harder--Narasimhan
polytope, that is, the convex hull in the Grothendieck group
$K(\Lambda\text{-mod})\otimes_{\mathbb Z}\mathbb R$ of the
dimension-vectors of all the submodules of $M$. Now the representation
spaces $\rep(\Lambda,\nu)$ of $\Lambda$ (one for each dimension-vector
$\nu$) are algebraic varieties, and it makes sense to speak of a
general point in an irreducible component of these varieties.
The Mirković--Vilonen polytopes then appear as the Harder--Narasimhan
polytopes of general $\Lambda$-modules.

The fact that the moment polytopes of Mirković and Vilonen's cycles
coincide with the Harder--Narasimhan polytopes of general
$\Lambda$-modules seems miraculous. The initial proof of this coincidence
was indeed rather circuitous. On the one hand, the moment polytopes were
described using Berenstein, Fomin and Zelevinsky's Chamber Ansatz and
Braverman and Gaitsgory's crystal structure on the set of Mirković and
Vilonen's cycles~\cite{Kamnitzer07,Kamnitzer10}. On the other hand, the
Harder--Narasimhan polytopes were studied with the help of tilting theory
on $\Lambda$-modules \cite{BaumannKamnitzer,BaumannKamnitzerTingley}.
The two descriptions were then equated thanks to Saito's work on
Lusztig data and reflections in the crystal $B(\infty)$~\cite{Saito}.

In this paper we propose a more straightforward argument.
Through the Geometric Satake Equivalence, the fundamental class in
cohomology of a Mirković--Vilonen cycle $Z$ can be viewed as a basis
element $b_Z$ in an irreducible rational $G$-module, which we view as
embedded in $\mathscr O(N)$. Gathering the elements $b_Z$ coming from
all the irreducible rational $G$-modules, we obtain a basis of
$\mathscr O(N)$, called the Mirković--Vilonen
basis~\cite{BaumannKamnitzerKnutson}.
Similarly, the study of $\Lambda$-modules leads to a basis of
$\mathscr O(N)$, called the dual semicanonical
basis~\cite{GeissLeclercSchroer05,Lusztig00}.
This basis consists of elements $\rho_Y$, where $Y$ denotes
an irreducible component of a variety $\rep(\Lambda,\nu)$.
Lastly, to a function $f\in\mathscr O(N)$, homogeneous with respect
to the weight grading, we associate $\underline\Pol(f)$, the convex hull
in $P\otimes_{\mathbb Z}\mathbb R$ of the set of weights of the
$N$-submodule generated by $f$. Said differently, in terms of the Hopf
algebra structure of~$\mathscr O(N)$, the polytope $\underline\Pol(f)$
is the convex hull of the weights of the elements $b_i$ that appear in
a minimal writing $\sum_{i=1}^nb_i\otimes c_i$ of the coproduct of $f$.

We will show that: (1) For each irreducible component $Y$ of
$\rep(\Lambda,\nu)$, the Harder--Narasimhan polytope of a general
point $M$ of $Y$ is equal to $-\underline\Pol(\rho_Y)$.
(2) For each Mirković--Vilonen cycle $Z$, the moment polytope of
$Z$ is equal to $\underline\Pol(b_Z)$.
(3) Both the dual semicanonical basis and the Mirković--Vilonen
basis are compatible with the filtration $K\mapsto\mathscr S(K)$,
in the sense that each subspace $\mathscr S(K)$ is spanned by subsets
of either basis.

By construction, $\mathscr S(K)$ is the vector space spanned by the
functions $f$ such that $\underline\Pol(f)\subset K$. Then (3)
implies that the jumps in the filtration $K\mapsto\mathscr S(K)$
occur precisely at the polytopes $\underline\Pol(\rho_Y)$ and
$\underline\Pol(b_Z)$, so these two families of polytopes are the
same, up to the indexing.

The proof of (1) is almost contained in \cite{BaumannKamnitzerTingley}:
each vertex $\mu$ of the Harder--Narasimhan polytope of a $\Lambda$-module
$M$ is the dimension-vector of a unique submodule $M_t$ of $M$.
Moreover, if $M$ is a general point of an irreducible component $Y$
of $\rep(\Lambda,\nu)$, then $M_t$ and $M/M_t$ are general points
in irreducible components $Y_t$ and $Y_f$ of $\rep(\Lambda,\mu)$
and $\rep(\Lambda,\nu-\mu)$, respectively.
One can then easily show that the term $\rho_{Y_t}\otimes\rho_{Y_f}$
occurs in the coproduct of $\rho_Y$, which implies that $-\mu$ belongs
to $\underline\Pol(\rho_Y)$. Assertion (2) is proved by a similar
argument: one can cut a Mirković--Vilonen cycle at each vertex of its
moment polytope, thereby producing two other cycles $Z_t$ and $Z_s$
such that $b_{Z_t}\otimes b_{Z_s}$ occurs in the coproduct of~$b_Z$
(see section~\ref{ss:MVPolite} for the accurate statement). These
observations can also be put to good use to prove (3) (see the proofs
of Proposition~\ref{pr:SCPolite} and Theorem~\ref{th:MVPolite}).

The dual semicanonical basis and the Mirković--Vilonen basis are
examples of perfect bases of $\mathscr O(N)$. This notion, introduced
by Berenstein and Kazhdan~\cite{BerensteinKazhdan}, codifies a specific
behavior that a basis exhibits regarding the action of the Chevalley
generators of the Lie algebra of~$N$. The perfect bases of
$\mathscr O(N)$ are parameterized by a common combinatorial object,
namely Kashiwara's crystal~$B(\infty)$, and it is possible to recover
the Mirković--Vilonen polytopes from the crystal operations on
$B(\infty)$~\cite{Kamnitzer07}. Perfect bases enjoy a number of nice
properties (see for instance \cite{BaumannKamnitzerKnutson} for a
review), but are generally not compatible with the subspaces
$\mathscr S(K)$. However, as we saw in (3) above, the dual
semicanonical basis and Mirković and Vilonen's
basis are compatible with them, and as it happens, the dual of
Lusztig's canonical basis~\cite{Lusztig93} (Kashiwara's upper global
basis~\cite{Kashiwara}) also is. It then seems desirable to record
this feature in a definition. With this aim in view, we shall
introduce the notion of polite basis. Polite bases are perfect in
the sense of Berenstein and Kazhdan, but the converse is not true.

As mentioned above, if $f$ is a homogeneous element in
$\mathscr O(N)$ with respect to the weight grading and if
$\Delta(f)=\sum_{i=1}^nb_i\otimes c_i$ is a minimal writing of the
coproduct of $f$, then $\underline\Pol(f)$ is the convex hull of the
degrees of the elements $b_i$. With this notation, a basis $\mathbf B$
is polite if each subspace $\mathscr S(K)$ is spanned by its
intersection with $\mathbf B$ and if, for each $f\in\mathbf B$ and
each vertex $\mu$ of $\underline\Pol(f)$, there is just one term
$b_i\otimes c_i$ in $\Delta(f)$ such that $b_i$ has weight $\mu$,
and for this $i$ both $b_i$ and $c_i$ belong to $\mathbf B$. The
precise definition we adopt in section~\ref{ss:PoliteBases} combines the
two conditions at the price of seeming intricate. One reason for
this choice is that the filtration is more conveniently described
when we regard polytopes as intersections of closed half-spaces.
In other words, we look at the polytopes through a plane projection
of the weight lattice.

The plan of this paper is the following. Section~\ref{se:Stability}
adapts Harder and Narasimhan's notion of stability to the case of
a connected bialgebra graded by a submonoid $\Gamma$ of the half-plane;
its main result is Theorem~\ref{th:Factorization}.
In section~\ref{se:PolytBases}, we apply this theory to
$\mathscr O(N)$ and define the notion of polite basis
(Definition~\ref{de:PoliteBasis}); we prove that the semicanonical
basis is polite (Proposition~\ref{pr:SCPolite}) and that polite
bases are perfect (Proposition~\ref{pr:Perfectness}).
Section~\ref{se:MVBasis} deals with Mirković and Vilonen's basis;
in particular we explain how to cut a Mirković--Vilonen cycle at a
vertex of its moment polytope (Proposition~\ref{pr:CuttingMVCycles})
and we prove that this basis is polite (Theorem~\ref{th:MVPolite}).

\textit{Acknowledgements.}
The work reported in this paper originates from a stay at Osaka Central
Advanced Mathematical Institute (OCAMI) in 2016. The author thanks this
institution, especially Yoshiyuki~Kimura for the invitation and the
useful discussions which ensued. He is also grateful to Joel Kamnitzer
for more than a decade of conversations on this topic and for comments
which helped improve the writing of this paper. Thanks are also due
to Frédéric Chapoton, JiaRui Fei and Dinakar Muthiah for relevant
suggestions.

\section{Stability for graded connected bialgebras}
\label{se:Stability}
In this section we consider a discrete submonoid $\Gamma$ of the right
half-plane
$$\bigl\{(r,d)\in\mathbb R^2\bigm|\text{$r>0$ or
($r=0$ and $d\geq0$)}\bigr\}$$
and a connected $\Gamma$-graded bialgebra
$A=\bigoplus_{\gamma\in\Gamma}A^\gamma$ over a field $k$. The
connectedness assumption means that $A^0=k$, where $0=(0,0)$ is the
identity element of~$\Gamma$.

Our aim is to explain a procedure to factorize $A$, patently inspired by
the method that Harder and Narasimhan introduced to study vector bundles
on a curve (see for instance \cite{Andre} for a survey of slope filtrations
and stability).

We denote the comultiplication and the augmentation maps of $A$ by
$\Delta$ and $\varepsilon$, respectively. If $a$ is a homogeneous element
in $A$, we write $|a|=\gamma$ to indicate that $a\in A^\gamma$.
Given $(\alpha,\beta)\in\Gamma^2$, the notation $\alpha-\beta$ will be
reserved to the case where the result of the subtraction in $\mathbb R^2$
actually belongs to $\Gamma$.

\subsection{Semistable elements}
\label{ss:SemistableElements}
For any homogeneous element $a\in A$, we denote by $L(a)$ (respectively,
$R(a)$) the set of all $\beta$ (respectively, $\gamma$) such that
$\Delta(a)$ has a nonzero component along a summand
$A^\beta\otimes A^\gamma$ in the decomposition
$$A\otimes A=\bigoplus_{(\beta,\gamma)\in\Gamma^2}
\bigl(A^\beta\otimes A^\gamma\bigr).$$
The obvious relation $R(a)=|a|-L(a)$ connects these two sets. If we
write $\Delta(a)=\sum_{i=1}^nb_i\otimes c_i$ with each $b_i$ and $c_i$
homogeneous and $n$ minimal, then
$$L(a)=\{|b_1|,\ldots,|b_n|\}\quad\text{and}\quad
R(a)=\{|c_1|,\ldots,|c_n|\}.$$
Since $A$ is connected (and assuming that $a\neq0$), we have
$\{0,|a|\}\subset L(a)$; specifically the homogeneous components of
bidegree $(0,|a|)$ (respectively, $(|a|,0)$) is $1\otimes a$
(respectively, $a\otimes1$).

\begin{lemma}
\label{le:LbiRci}
Let $a\in A$ be a homogeneous element and write
$\Delta(a)=\sum_{i=1}^nb_i\otimes c_i$ as above. Then
$L(b_i)\subset L(a)$ and $R(c_i)\subset R(a)$ for all $i\in\inter1n$.
\end{lemma}

\begin{proof}
Both families $(b_i)_{i\in\inter1n}$ and $(c_i)_{i\in\inter1n}$ consist of
linearly independent elements. Pick $i\in\inter1n$ and choose a linear form
$f:A\to k$ that vanishes on $c_j$ for all $j\neq i$ and that verifies
$f(c_i)=1$. The coassociativity axiom implies
\begin{align*}
\Delta(b_i)=(\Delta\otimes f)\circ\Delta(a)
&=(\id_A\otimes\id_A\otimes f)\circ(\Delta\otimes\id_A)\circ\Delta(a)\\
&=(\id_A\otimes((\id_A\otimes f)\circ\Delta))\circ\Delta(a)\\
&=\sum_{j=1}^nb_j\otimes((\id_A\otimes f)\circ\Delta(c_j)),
\end{align*}
which shows that $L(b_i)\subset L(a)$. The inclusion $R(c_i)\subset R(a)$
is proved in a similar way.
\end{proof}

We set $\overline{\mathbb R}=\mathbb R\cup\{\infty\}$; here $\infty$ is
regarded as larger than any real number. We define
$$\Gamma_\infty=\{(r,d)\in\Gamma\mid r=0\}\quad\text{and}\quad
\Gamma_\mu=\{(r,d)\in\Gamma\mid d=\mu r\}$$
for $\mu\in\mathbb R$. Given $\mu\in\overline{\mathbb R}$, we set
$$\Gamma_{\leq\mu}=\bigcup_{\nu\leq\mu}\Gamma_\nu\quad\text{and}\quad
\Gamma_{\geq\mu}=\bigcup_{\nu\geq\mu}\Gamma_\nu.$$

Let $\mu\in\overline{\mathbb R}$. A homogeneous element $a\in A$ is said
to be of slope $\mu$ if $|a|\in\Gamma_\mu$. It is said to be semistable of
slope $\mu$ if moreover $L(a)\subset\Gamma_{\leq\mu}$, or equivalently
$R(a)\subset\Gamma_{\geq\mu}$.

We denote by $A_\mu$ (respectively, $A_{[\mu]}$) the subspace of $A$
spanned by homogeneous elements of slope $\mu$ (respectively, semistable
elements of slope $\mu$). We denote by $p_\mu:A\to A_\mu$ the projection
according to the decomposition
$$A=A_\mu\oplus\left(\bigoplus_{\gamma\in\Gamma\setminus\Gamma_\mu}
A^\gamma\right)$$
and we set $\Delta_\mu=(p_\mu\otimes p_\mu)\circ\Delta$.

\begin{proposition}
\label{pr:StableBialgebras}
The subspace $A_{[\mu]}$ is a subalgebra of $A_\mu$. Further,
$\Delta_\mu(A_{[\mu]})\subset A_{[\mu]}\otimes A_{[\mu]}$ and $A_{[\mu]}$
becomes a bialgebra when endowed with the coproduct $\Delta_\mu$.
\end{proposition}

\begin{proof}
Set $A_{\leq\mu}=\sum_{\nu\leq\mu}A_\nu$ and $A_{\geq\mu}=\sum_{\nu\geq\mu}
A_\nu$. Then $A_\mu$, $A_{\leq\mu}$ and $A_{\geq\mu}$ are
subalgebras of $A$ and $p_\mu$ restricts to morphisms of algebras
$A_{\leq\mu}\to A_\mu$ and $A_{\geq\mu}\to A_\mu$. It follows that
$A_{[\mu]}=A_\mu\cap\Delta^{-1}(A_{\leq\mu}\otimes A_{\geq\mu})$
is a subalgebra of~$A$ and that the composition
$$A_{[\mu]}\xrightarrow\Delta A_{\leq\mu}\otimes A_{\geq\mu}
\xrightarrow{p_\mu\otimes p_\mu}A_\mu\otimes A_\mu$$
is a morphism of algebras.

Now let $a\in A$ be a homogeneous semistable element of slope $\mu$ and
write $\Delta(a)=\sum_{i=1}^nb_i\otimes c_i$ with each $b_i$ and $c_i$
homogeneous and $n$ minimal. By Lemma~\ref{le:LbiRci}
$$L(b_i)\subset L(a)\subset\Gamma_{\leq\mu}\quad\text{and}\quad
R(c_i)\subset R(a)\subset\Gamma_{\geq\mu}$$
for each $i\in\inter1n$, so $b_i$ and $c_i$ are semistable as soon as
they are of slope~$\mu$. Therefore
$\Delta_\mu(a)\in A_{[\mu]}\otimes A_{[\mu]}$. Routine verifications show
then that $A_{[\mu]}$ equipped with the coproduct $\Delta_\mu$ fulfills
the requirements to be a bialgebra.
\end{proof}

\subsection{Ordered Monomials}
\label{ss:Factorization}
For any homogeneous element $a\in A$ of nonzero degree, we denote the convex
hull of $L(a)$ in $\mathbb R^2$ by $\Pol(a)$. The (closed) region between
the upper rim of $\Pol(a)$---which starts at $0$ and ends at $|a|$---and
the straight line between these two points will be of particular interest;
we denote it by $\Pol^\wedge(a)$. Obviously $a$ is semistable if and only
if $\Pol^\wedge(a)$ is just the straight line between $0$ and $|a|$.

\begin{other}{Definition}
An \textit{ordered monomial} is a product $a_1a_2\cdots a_p$ of homogeneous
semistable elements of nonzero degree, ordered so that their slopes form
a decreasing sequence.
\end{other}

\begin{lemma}
\label{le:OMPolygon}
Let $a=a_1a_2\cdots a_p$ be an ordered monomial.
\begin{enumerate}
\item
The upper rim of the polytope $\Pol(a)$ is the polygonal line
$\polyline{0,|a_1|,|a_1|+|a_2|,\ldots,|a|}$.
\item
Let $j\in\inter0p$ and let $\alpha=|a_1|+\cdots+|a_j|$ be the $j$-th
extremal point on the upper rim of $\Pol(a)$. Then the homogeneous
component of $\Delta(a)$ of bidegree $(\alpha,|a|-\alpha)$ is
$a_1\cdots a_j\otimes a_{j+1}\cdots a_p$.
\end{enumerate}
\end{lemma}

\begin{proof}
As the element $a_j$ is homogeneous and semistable, the upper rim
of $\Pol(a_j)$ is the line $\polyline{0,|a_j|}$. The condition on
the slopes then entails that the upper rim of the Minkowski sum
$\Pol(a_1)+\cdots+\Pol(a_p)$ is the polygonal line
$\polyline{0,|a_1|,|a_1|+|a_2|,\ldots,|a|}$. Furthermore, given
$j\in\inter0p$, the only tuple
$(\alpha_1,\ldots,\alpha_p)\in\Pol(a_1)\times\cdots\times\Pol(a_p)$
such that $\alpha_1+\cdots+\alpha_p=|a_1|+\cdots+|a_j|$ is
$(\alpha_1,\ldots,\alpha_p)=(|a_1|,\ldots,|a_j|,0,\ldots,0)$.
The lemma follows from these geometric considerations, using the
multiplicativity of the coproduct and the connectedness of $A$.
\end{proof}

Ordered monomials span the algebra $A$. More precisely:

\begin{proposition}
\label{pr:Expansion}
Any homogeneous element $a$ of nonzero degree is the sum of ordered
monomials $a_1$, \dots, $a_\ell$ such that $|a_i|=|a|$ and
$\Pol^\wedge(a_i)\subset\Pol^\wedge(a)$ for each $i\in\inter1\ell$.
\end{proposition}

The proof uses a method introduced by Schiffmann \cite{Schiffmann}.

\begin{proof}
The statement is obviously true if $a$ itself is already semistable.
In the rest of the proof, we assume that $a$ is not semistable (in
particular it is not of slope $\infty$) and proceed by induction on
the number $N(a)$ of points in $\Gamma\cap\Pol^\wedge(a)$.

Let $r$ be the abscissa of $|a|$. The upper rim of the polytope $\Pol(a)$
is the graph of a piecewise linear concave function $f:[0,r]\to\mathbb R$
and has at least one angular point. We choose such a point, let us say
$\alpha$; it belongs to $L(a)$.

We write $\Delta(a)=\sum_{i=1}^nb_i\otimes c_i$ with each $b_i$ and $c_i$
homogeneous and $n$ minimal. Let $I\subset\inter1n$ be the set of all
indices $i$ such that $|b_i|=\alpha$ and set $a'=a-\sum_{i\in I}b_ic_i$.
We will show that $N(b_i)<N(a)$ and $N(c_i)<N(a)$ for all $i\in I$ and
that $N(a')<N(a)$. Applying the induction hypothesis, we may then write
these elements $a'$, $b_i$ and $c_i$ as sums of ordered monomials and
obtain thereby the desired expression for $a$.

Let $x$ denote the abscissa of $\alpha$. Let us pick $i\in I$. By
Lemma~\ref{le:LbiRci} we have $L(b_i)\subset L(a)$. Therefore $\Pol(b_i)$
is contained in $\Pol(a)\cap([0,x]\times\mathbb R)$. Noting that
the line $\mathbb R|b_i|$ lies above $\mathbb R|a|$ in the band
$[0,x]\times\mathbb R$, we obtain $\Pol^\wedge(b_i)\subset\Pol^\wedge(a)$
and conclude that $N(b_i)<N(a)$ (the inequality is strict because the
point $|a|$ is in $\Pol^\wedge(a)$ but not in $\Pol^\wedge(b_i)$).
Likewise, $R(c_i)\subset R(a)$. Since $|a|-|c_i|=|b_i|=\alpha$, this can
be rewritten as $\alpha+L(c_i)\subset L(a)$. Therefore $\alpha+\Pol(c_i)$
is contained in $\Pol(a)\cap([x,r]\times\mathbb R)$. Noting that the
line $\alpha+\mathbb R|c_i|$ lies above $\mathbb R|a|$ in the band
$[x,r]\times\mathbb R$, we obtain
$\alpha+\Pol^\wedge(c_i)\subset\Pol^\wedge(a)$ and conclude that
$N(c_i)<N(a)$.

Now let $\beta\in L(b_i)$ and $\gamma\in L(c_i)$. Both $\beta$ and
$\alpha+\gamma$ belong to $\Pol(a)$. Let $y$ and $z$ denote the abscissae
of these two points; then $0\leq y\leq x\leq z\leq r$. Also, $\alpha$ is
the point $(x,f(x))$, $\beta$ lies below the point $(y,f(y))$, and
$\alpha+\gamma$ lies below the point $(z,f(z))$, so $\beta+\gamma$ lies
below the point $(y+z-x,f(y)+f(z)-f(x))$. Since $f$ is concave, we have
$f(y)+f(z)\leq f(x)+f(y+z-x)$, and therefore $\beta+\gamma$ lies below
$(y+z-x,f(y+z-x))$, that is, below the upper rim of $\Pol(a)$. To sum up,
all the points in $L(b_i)+L(c_i)$ lie below the upper rim of $\Pol(a)$.
Moreover the concavity inequality is strict if $(y,z)\neq(x,x)$ (because
$\alpha$ is an extremal point of $\Pol(a)$), so $(\alpha,0)$ is the only
pair $(\beta,\gamma)\in L(b_i)\times L(c_i)$ for which $\beta+\gamma$
lies on the upper rim of $\Pol(a)$.

Since $\Delta$ is a morphism of algebras, we have
$L(b_ic_i)\subset L(b_i)+L(c_i)$. From the previous paragraph, we deduce
that for each $i\in I$, all the points in $L(b_ic_i)$ lie below the upper
rim of $\Pol(a)$, and that moreover the homogeneous component of degree
$(\alpha,|a|-\alpha)$ in $\Delta(b_ic_i)$ is $b_i\otimes c_i$. We conclude
that all the points in $L(a')$ lie below the upper rim of $\Pol(a)$ and
that $\Delta(a')$ does not contain any term of degree $(\alpha,|a|-\alpha)$.
In other words $\Pol^\wedge(a')$ is contained in
$\Pol^\wedge(a)\setminus\{\alpha\}$. It follows that $N(a')<N(a)$,
as announced.

Now pick $i\in I$, and consider two concave polygonal lines, one that joins
$0$ to $\alpha$ and is contained in $\Pol^\wedge(b_i)$, and another one
that joins $\alpha$ to $|a|$ and is contained in $\alpha+\Pol^\wedge(c_i)$.
Arguments given earlier in the proof imply that the concatenation of these
two lines lies in $\Pol^\wedge(a)$; in addition, this concatenated line is
concave because $\alpha$ is an extremal point of $\Pol(a)$. These
geometrical observations guarantee that, when we substitute in
$a'+\sum_{i\in I}b_ic_i$ expansions of $a'$, $b_i$ and $c_i$ obtained by
induction, we obtain an expression for $a$ that satisfies the requirements.
\end{proof}

Let $J\subset\overline{\mathbb R}$. For any finite subset $M\subset J$,
we can form the vector space $\bigotimes_{\mu\in M}A_{[\mu]}$, where the
tensor product is computed in the decreasing order. An inclusion
$M_1\subset M_2\subset J$ gives rise to an injective linear map
$\bigotimes_{\mu\in M_1}A_{[\mu]}\to\bigotimes_{\mu\in M_2}A_{[\mu]}$
defined by inserting the unit element $1$ at each factor that occurs only
in the second product. We denote the limit of this inductive system by
$\bigotimes_{\mu\in J}A_{[\mu]}$. Inclusions
$J_1\subset J_2\subset\overline{\mathbb R}$ give rise to injective
linear maps $\bigotimes_{\mu\in J_1}A_{[\mu]}\to\bigotimes_{\mu\in J_2}
A_{[\mu]}\to\bigotimes_{\mu\in\overline{\mathbb R}}A_{[\mu]}$.

\begin{theorem}
\label{th:Factorization}
The multiplication in $A$ induces a bijective linear map
$m:\bigotimes_{\mu\in\overline{\mathbb R}}A_{[\mu]}\to A$.
\end{theorem}

\begin{proof}
The surjectivity of $m$ is a direct consequence of Proposition
\ref{pr:Expansion}.

Suppose that there exists a nonzero element
$$a=\sum_{i=1}^na_{i,1}\otimes\cdots\otimes a_{i,p_i}$$
in the kernel of $m$. In this writing, each $a_{i,j}$ is homogeneous
semistable of nonzero degree, and for each $i\in\inter1n$ the slope of
$a_{i,j}$ decreases when $j$ increases from $1$ to~$p_i$.

We assume that $a$ and this writing have been chosen with $n$ as small
as possible, and for this $n$, with $\max(p_1,\ldots,p_n)$ as small as
possible. This minimality assumption implies that $a$ is homogeneous and
that all the terms in the sum above have the same degree. Also, at most
one $p_i$ is equal to $1$ (all the terms with $p_i=1$ have the same
slope so they can be combined).

Let $\mu$ be the maximum value among the slopes of the elements $a_{i,1}$,
let $\alpha$ be the point that is the farthest from the origin among
those $\bigl|a_{i,1}\bigr|$ that lie on the line $\mathbb R\mu$, and
let $I\subset\inter1n$ be the subset of indices $i$ such that
$\bigl|a_{i,1}\bigr|=\alpha$. Extracting the homogeneous component
of degree $(\alpha,|a|-\alpha)$ from the equation
$$\Delta\Biggl(\sum_{i=1}^na_{i,1}\cdots a_{i,p_i}\Biggr)=0$$
with the help of Lemma~\ref{le:OMPolygon}, we obtain
$$\sum_{i\in I}a_{i,1}\otimes a_{i,2}\cdots a_{i,p_i}=0.$$
Any linear form $f:A^\alpha\to k$ then gives birth to an element
$$\sum_{i\in I}f\bigl(a_{i,1}\bigr)
a_{i,2}\otimes\cdots\otimes a_{i,p_i}$$
in the kernel of $m$. Our minimality requirement for $a$ forces this
shorter element to be zero.

Picking an $f$ that does not annihilate all the $a_{i,1}$,
we obtain a linear dependence relation between the elements
$a_{i,2}\otimes\cdots\otimes a_{i,p_i}$. With the help of this
relation, we can shorten the writing of $a$, and thereby contradict our
minimality requirement. Consequently our assumption about the existence
of $a$ was incorrect, and we conclude that $m$ is injective.
\end{proof}

The next proposition encompasses the existence of straightening relations
in $A$.
\begin{proposition}
\label{pr:IntervalSubalgebras}
If $J$ is an interval in $\overline{\mathbb R}$, then
$m\Bigl(\bigotimes_{\mu\in J}A_{[\mu]}\Bigr)$ is a subalgebra of $A$.
\end{proposition}

\begin{proof}
Let $a=a_1a_2\cdots a_p$ be a product of homogeneous semistable elements
whose slopes are in $J$. Let $\sigma$ be a permutation of $\inter1p$
that reorder the sequence $(a_1,a_2,\ldots,a_p)$ by weakly decreasing
slopes. Similar arguments as in Lemma~\ref{le:OMPolygon} prove that the
upper rim of $\Pol(a)$ is the polygonal line $\polyline{0,|a_{\sigma(1)}|,
|a_{\sigma(1)}|+|a_{\sigma(2)}|,\ldots,|a|}$. Applying Proposition
\ref{pr:Expansion}, we expand $a$ as a sum of ordered monomials $b_i$
such that $|b_i|=|a|$ and $\Pol^\wedge(b_i)\subset\Pol^\wedge(a)$ for
each index~$i$. The slopes of all the semistable factors of $b_i$ are
then less than or equal to the slope of $a_{\sigma(1)}$ and greater or
equal than the slope of $a_{\sigma(p)}$, so are in $J$. We conclude that
$a$ belongs to $m\Bigl(\bigotimes_{\mu\in J}A_{[\mu]}\Bigr)$. As a
particular case, we see that this space is stable by multiplication.
\end{proof}

Given $\mu\in\mathbb R$, we define the subalgebras
$A_{[\leq\mu]}=m\Bigl(\bigotimes_{\nu\leq\mu}A_{[\nu]}\Bigr)$ and
$A_{[\geq\mu]}=m\Bigl(\bigotimes_{\nu\geq\mu}A_{[\nu]}\Bigr)$.

\begin{proposition}
\label{pr:Coideals}
The subalgebra $A_{[\leq\mu]}$ (respectively, $A_{[\geq\mu]}$) is a right
(respectively, left) coideal of $A$.
\end{proposition}

\begin{proof}
Lemma~\ref{le:OMPolygon} and Proposition~\ref{pr:Expansion} imply that
a homogeneous element $a$ belongs to $A_{[\leq\mu]}$ (respectively,
$A_{[\geq\mu]}$) if and only if $L(a)\subset\Gamma_{\leq\mu}$
(respectively, $R(a)\subset\Gamma_{\geq\mu}$). The proposition then
follows from Lemma~\ref{le:LbiRci}.
\end{proof}

Likewise we define subalgebras $A_{[<\mu]}$ and $A_{[>\mu]}$.

\subsection{The splitting isomorphisms}
\label{ss:Splitting}
We will need to look at the multiplication map
$A_{[>0]}\otimes A_{[0]}\otimes A_{[<0]}\to A$
through the prism of filtrations.

Let $\Pi'=\{(r,d)\in\mathbb R^2\mid\text{$d<0$ or ($d=0$ and $r\geq0$)}\}$.
For $(\alpha_1,\alpha_2)\in\Gamma^2$, we write $\alpha_1\leq'\alpha_2$ if
$\alpha_1+\Pi'\subset\alpha_2+\Pi'$. This relation $\leq'$ is a total
order on $\Gamma$.

We define an increasing filtration $F'_\bullet A$ of the vector space
$A$ by the poset $(\Gamma,\leq')$ as follows: for $\alpha\in\Gamma$,
the subspace $F'_\alpha A$ is the linear span of all homogeneous
elements $a\in A$ such that $L(a)\subset\alpha+\Pi'$. This filtration
is compatible with the algebra structure of $A$, and the associated
graded algebra $\gr'_\bullet A$ is bigraded by $\Gamma$, for the
original $\Gamma$-grading of $A$ comes through.

As an example, we see that $F'_0A=A_{[\leq0]}$ (see the proof of
Proposition~\ref{pr:Coideals}) and $F'_\alpha A=0$ if $\alpha<'0$
(if $F'_\alpha A$ contains a nonzero homogeneous element $a$, then
$0\in L(a)$ and $L(a)\subset\alpha+\Pi'$, so $0\leq'\alpha$),
therefore $\gr'_0A=A_{[\leq0]}$. Similarly, we see that
$A_{[>0]}=\bigoplus_{\gamma\in\Gamma}(F'_\gamma A)^\gamma$.

\begin{proposition}
\label{pr:Splitting'}
There is a unique linear map $\overline\Delta{}':\gr'_\bullet A\to
A_{[>0]}\otimes A_{[\leq0]}$ that sends a homogeneous element
$a\in\gr'_\alpha A$ to the homogeneous component of degree
$(\alpha,|a|-\alpha)$ of $\Delta(a)$. This map $\overline\Delta{}'$ is
an isomorphism of algebras. The inverse map is induced by the
multiplication in $A$.
\end{proposition}

\begin{proof}
Let $\alpha\in\Gamma$. The vector space $F'_\alpha A$ is spanned by the
homogeneous elements $a$ that satisfy $\Pol^\wedge(a)\subset\alpha+\Pi'$,
so is spanned by the ordered monomials that satisfy this same condition
(Proposition~\ref{pr:Expansion}). Likewise, the space
$\sum_{\beta<'\alpha}F'_\beta A$ is spanned by the ordered monomials
that satisfy $\Pol^\wedge(a)\subset(\alpha+\Pi')\setminus\{\alpha\}$.
In addition, if an ordered monomial $a=a_1a_2\cdots a_p$ belongs to
$F'_\alpha A$, if $\alpha\in L(a)$, and if $j$ is the largest index
in $\inter1p$ such that $a_j$ has positive slope, then
$|a_1|+|a_2|+\cdots+|a_j|=\alpha$ and the homogeneous
component of degree $(\alpha,|a|-\alpha)$ of $\Delta(a)$
is $a_1\cdots a_j\otimes a_{j+1}\cdots a_p$.
These remarks entail that the map $\overline\Delta{}'$ is well-defined.

In the other direction, if $a$ is a homogeneous element in $A_{[>0]}$ of
degree $\alpha$ and if $b$ is in $A_{[\leq0]}$, then $ab$ belongs to
$F'_\alpha A$. Let
$\overline m:A_{[>0]}\otimes A_{[\leq0]}\to\gr'_\bullet A$ be the map
that sends $a\otimes b$ to the image of $ab$ in $\gr'_\alpha A$, where
$a$ and $b$ are as in the previous sentence. Then $\overline\Delta'$
and $\overline m$ are mutually inverse linear bijections. Further,
$\overline\Delta'$ inherits from the coproduct the property of being
an algebra morphism.
\end{proof}

Let $\Pi''=\{(r,d)\in\mathbb R^2\mid\text{$d>0$ or ($d=0$ and $r\geq0$)}\}$.
For $(\beta_1,\beta_2)\in\Gamma^2$, we write $\beta_1\leq''\beta_2$ if
$\beta_1+\Pi''\subset\beta_2+\Pi''$. This relation $\leq''$ is a total
order on $\Gamma$.

We define an increasing filtration $F''_\bullet A$ of the algebra $A$ by
the poset $(\Gamma,\leq'')$ as follows: for $\beta\in\Gamma$, the space
$F''_\beta A$ is the linear span of all homogeneous elements $a\in A$
such that $R(a)\subset\beta+\Pi''$. The associated graded algebra
$\gr''_\bullet A$ is then bigraded by $\Gamma$.

\begin{proposition}
\label{pr:Splitting''}
There is a unique linear map $\overline\Delta{}'':\gr''_\bullet A\to
A_{[\geq0]}\otimes A_{[<0]}$ that sends a homogeneous element
$a\in\gr''_\beta A$ to the homogeneous component of degree
$(|a|-\beta,\beta)$ of $\Delta(a)$. This map $\overline\Delta{}''$ is
an isomorphism of algebras. The inverse map is induced by the
multiplication in $A$.
\end{proposition}

The filtration $F'_\bullet$ on $A$ induces a filtration on
$\gr''_\bullet A$, and likewise the filtration $F''_\bullet$ induces
a filtration on $\gr'_\bullet A$. The associated graded algebras
$\gr'_\bullet\gr''_\bullet A$ and $\gr''_\bullet\gr'_\bullet A$ are
naturally isomorphic and are trigraded by $\Gamma$.

\begin{proposition}
\label{pr:Slicing}
There is a unique linear map
$\overline\Delta_2:\gr'_\bullet\gr''_\bullet A\to A_{[>0]}
\otimes A_{[0]}\otimes A_{[<0]}$ that sends a homogeneous element
$a\in\gr'_\alpha\gr''_\beta A$ to the homogeneous component of degree
$(\alpha,|a|-\alpha-\beta,\beta)$ of $(\Delta\otimes\id)\circ\Delta(a)$.
This map $\overline\Delta_2$ is an isomorphism of algebras. The inverse
map is induced by the multiplication in $A$.
\end{proposition}

We omit the proofs of the last two propositions.
For future use, we record that $A_{[\geq0]}=F''_0A=\gr''_0A$, that
$A_{[<0]}=\bigoplus_{\gamma\in\Gamma}(F''_\gamma A)^\gamma$, and that
$A_{[0]}=F'_0A\cap F''_0A=\gr'_0\gr''_0A$.

\subsection{Duality}
\label{ss:Duality}
In this section we assume that each homogeneous component $A^\gamma$ is
finite-dimensional. The graded dual of $A$, which we denote by $B$, is also
a $\Gamma$-graded connected bialgebra. We use the same notation $\Delta$ for
the coproduct of $B$ as for $A$. We can define the bialgebras $B_{[\mu]}$
(Proposition~\ref{pr:StableBialgebras}) and we can look at the basis of
ordered monomials in $B$.

\begin{proposition}
\label{pr:Duability}
\begin{enumerate}
\item
Let $\langle\,,\,\rangle:A\times B\to k$ be the duality bracket,
and let $a$ and $b$ be ordered monomials in $A$ and $B$, respectively.
Then $\langle a,b\rangle=0$ unless $\Pol^\wedge(a)=\Pol^\wedge(b)$.
\item
For each $\mu\in\overline{\mathbb R}$, the bialgebra $B_{[\mu]}$ identifies
with the graded dual of $A_{[\mu]}$.
\end{enumerate}
\end{proposition}

\begin{proof}
Let $a=a_1\cdots a_p$ and $b=b_1\cdots b_q$ be ordered monomials in $A$
and $B$, respectively. Suppose that $|a|=|b|$, but that
$\Pol^\wedge(a)\neq\Pol^\wedge(b)$. Exchanging $A$ and $B$ if necessary,
we may assume that one extremal point on the upper rim of $\Pol(a)$, say
$\alpha$, does not belong to $\Pol^\wedge(b)$. By Lemma~\ref{le:OMPolygon},
there exists $j\in\inter1{p-1}$ such that $\alpha=|a_1|+\cdots+|a_j|$.
As the homogeneous component of degree $(\alpha,|a|-\alpha)$ of $\Delta(b)$
is zero, we have
$$\langle a,b\rangle=\langle(a_1\cdots a_j)(a_{j+1}\cdots a_p),b\rangle
=0.$$
This shows the first assertion. The second one then follows from
Theorem~\ref{th:Factorization}.
\end{proof}

\section{Polytopes and bases of $\mathscr O(N)$}
\label{se:PolytBases}

\subsection{Notation}
\label{ss:Notation}
For the rest of the paper we consider a split connected reductive group $G$
over a field $k$ of characteristic $0$. We fix a Borel subgroup $B$ and
a maximal torus $T$ contained in $B$. Let $P$ be the character lattice of
$T$, let $\Phi$ be the root system of $(G,T)$, let $\Phi_+$ (respectively,
$\Phi_-$) be the set of positive (respectively, negative) roots determined
by $B$, and let $\{\alpha_i\}_{i\in I}$ be the set of simple roots.
We denote by $Q$ (respectively, $Q_\pm$) the subgroup (respectively,
submonoid) of $P$ generated by $\Phi$ (respectively, $\Phi_\pm$).

Let $N$ be the unipotent radical of $B$ and let $\mathfrak n$ be the Lie
algebra of $N$. For $i\in I$, we choose a root vector $e_i\in\mathfrak n$
of weight $\alpha_i$. The enveloping algebra $U(\mathfrak n)$ is generated
by the elements $e_i$. It is graded by $Q_+$, with $\deg e_i=\alpha_i$.
For $(i,\ell)\in I\times\mathbb N$, the divided power $e_i^\ell/\ell!$
in $U(\mathfrak n)$ is denoted by $e_i^{(\ell)}$.

Since $N$ is an algebraic group, the algebra $\mathscr O(N)$ of
regular functions on $N$ is a Hopf algebra. It acquires a grading
$\bigoplus_{\lambda\in Q_-}\mathscr O(N)_\lambda$ by means of the
conjugation action of $T$ on $N$. The group $N$ and its Lie algebra
$\mathfrak n$ acts on both sides on $\mathscr O(N)$; our convention
is that $n\cdot f=f(\bm?n)$ and $f\cdot n=f(n\bm?)$ for all
$(n,f)\in N\times\mathscr O(N)$. Denoting by $1_N$ the identity
element of $N$, we then have $(x\cdot f)(1_N)=(f\cdot x)(1_N)$ for all
$(x,f)\in U(\mathfrak n)\times\mathscr O(N)$, so it makes sense to
define a pairing between $U(\mathfrak n)$ and $\mathscr O(N)$ by
$\langle x,f\rangle=(x\cdot f)(1_N)$. This pairing is perfect and
identifies $\mathscr O(N)$ as a bialgebra to the graded dual of
$U(\mathfrak n)$. For $i\in I$, we denote by $\zeta_i\in\mathscr O(N)$
the function of weight $-\alpha_i$ such that $\langle e_i,\zeta_i\rangle=1$.

We define the height $\height\lambda$ of an element $\lambda\in Q$ as the
sum of the coordinates of $\lambda$ in the basis $\{\alpha_i\}_{i\in I}$.
We define an involution $*$ on the vector space $\mathscr O(N)_\lambda$
by setting $f^*(n)=(-1)^{\height\lambda}f(n^{-1})$ for all
$(n,f)\in N\times\mathscr O(N)_\lambda$. This involution $*$ is extended
linearly to the whole algebra $\mathscr O(N)$.

For a function $f\in\mathscr O(N)$ of weight $\lambda$, we denote by
$\underline L(f)$ the set of all weights $\mu$ such that the coproduct
$\Delta(f)$ has a nonzero component along the summand $\mathscr O(N)_\mu
\otimes\mathscr O(N)_{\lambda-\mu}$ of $\mathscr O(N)^{\otimes2}$.
In other words, $\underline L(f)$ is the set of weights of the left
$N$-submodule of $\mathscr O(N)$ generated by $f$. We set
$Q_{\mathbb R}=Q\otimes_{\mathbb Z}\mathbb R$ and we denote the convex
hull of $\underline L(f)$ in this vector space by $\underline\Pol(f)$.
This definition mimics the notion of KLR polytope introduced in
\cite{TingleyWebster}.

In the sequel, we will just write $U$ and $\mathscr O$ for respectively
$U(\mathfrak n)$ and $\mathscr O(N)$ when this does not cause any confusion.

\subsection{Stability in $U(\mathfrak n)$ and $\mathscr O(N)$}
\label{ss:StabilityON}
In our story the dual $Q_{\mathbb R}^*$ of $Q_{\mathbb R}$ plays the
role of the space of stability parameters.

Let us pick an element $\theta\in Q_{\mathbb R}^*$. The map
$\lambda\mapsto(\height\lambda,\theta(\lambda))$ projects $Q_+$ onto
a submonoid $\Gamma_\theta$ of the right half-plane. The $Q_+$-grading
on $U$ then descends to a $\Gamma_\theta$-grading, for which the results
in section~\ref{se:Stability} can be applied. Likewise the $Q_-$-grading
on $\mathscr O$ descends to a $\Gamma_\theta$-grading by means of the
map $J_\theta:\lambda\mapsto-(\height\lambda,\theta(\lambda))$.

For each slope $\mu\in\mathbb R$, the set of positive roots $\alpha$
such that $\theta(\alpha)/\height\alpha=\mu$ is closed in $\Phi$.
The corresponding root vectors span a subalgebra $\mathfrak n_\mu$ of
the Lie algebra $\mathfrak n$. As homogeneous elements in $\mathfrak n$
are primitive hence semistable, the subalgebra of $U$ generated by
$\mathfrak n_\mu$ is contained in $U_{[\mu]}$. Comparing then Theorem
\ref{th:Factorization} with the Poincaré--Birkhoff--Witt theorem, we
conclude that $U(\mathfrak n_\mu)=U_{[\mu]}$.

We set $\mathfrak n_{>0}=\sum_{\mu>0}\mathfrak n_\mu$ and
$\mathfrak n_{<0}=\sum_{\mu<0}\mathfrak n_\mu$, and we denote by $N_{>0}$,
$N_0$ and $N_{<0}$ the closed subgroups of $N$ with Lie algebras
$\mathfrak n_{>0}$, $\mathfrak n_0$ and $\mathfrak n_{<0}$, respectively.
Then $N$ is the Zappa--Szép product\footnote{We recall that a group
$N$ is the (internal) Zappa--Szép product of two subgroups $N'$ and
$N''$ (in symbols, $N=N'\bowtie N''$) if the product in $N$ induces
a bijection $N'\times N''\to N$.} $N_{>0}\bowtie N_0\bowtie N_{<0}$.

Given a subgroup $H$ of $N$, we denote by ${}^H\mathscr O$ (respectively,
$\mathscr O^H$) the subalgebras of invariant elements in $\mathscr O$
with respect to the left (respectively, right) action of $H$. With this
notation, we deduce  from Proposition~\ref{pr:Duability} the following
explicit description:
$$\mathscr O_{[\geq0]}={}^{N_{<0}}\mathscr O,\quad
\mathscr O_{[>0]}={}^{N_0N_{<0}}\mathscr O,\quad
\mathscr O_{[\leq0]}=\mathscr O^{N_{>0}},\quad
\mathscr O_{[<0]}=\mathscr O^{N_{>0}N_0},\quad
\mathscr O_{[0]}={}^{N_{<0}}\mathscr O^{N_{>0}}.$$
Since the decomposition $N=N_{>0}\bowtie N_0\bowtie N_{<0}$ is an
isomorphism of algebraic varieties, we obtain isomorphisms from
$\mathscr O_{[>0]}$, $\mathscr O_{[0]}$ and $\mathscr O_{[<0]}$ onto
the algebras of regular functions on $N_{>0}$, $N_0$ and $N_{<0}$,
respectively.

In the next sections, we will often decorate the objects with a
label $\theta$ when we want to stress that they depend on the choice
of a stability parameter. As an example, the filtrations
defined in section~\ref{ss:Splitting} will be denoted by
${}^\theta F'_\bullet\mathscr O$ and ${}^\theta F''_\bullet\mathscr O$.
Likewise, given a weight function $f\in\mathscr O$, the subsets $L(f)$,
$\Pol(f)$ and $\Pol^\wedge(f)$ defined by the $\Gamma_\theta$-grading of
$\mathscr O$ will be adorned with an index $\theta$, so for instance
$L_\theta(f)=J_\theta(\underline L(f))$.

\subsection{Polite bases}
\label{ss:PoliteBases}
The following definition highlights properties that the dual canonical
basis, the dual semicanonical basis and Mirković and Vilonen's basis
share.

\begin{other}{Definition}
\label{de:PoliteBasis}
A basis $\mathbf B$ of $\mathscr O$ is said to be \textit{polite} if it
satisfies the following four conditions:
\begin{enumerate}
\item
The elements of $\mathbf B$ are weight vectors.
\item
For each $(i,n)\in I\times\mathbb N$, the function $\zeta_i^n$ belongs
to $\mathbf B$.
\item
For each $\theta\in Q_{\mathbb R}^*$, the basis $\mathbf B$ is compatible
with the filtration ${}^\theta F'_\bullet\mathscr O$ and with the
isomorphism defined in Proposition~\ref{pr:Splitting'}.
\item
For each $\theta\in Q_{\mathbb R}^*$, the basis $\mathbf B$ is compatible
with the filtration ${}^\theta F''_\bullet\mathscr O$ and with the
isomorphism defined in Proposition~\ref{pr:Splitting''}.
\end{enumerate}
\end{other}

The compatibility required in condition (iii) means the following.
First, for each degree $\gamma\in\Gamma_\theta$, the subspace
${}^\theta F'_\gamma\mathscr O$ should be spanned by
$\mathbf B\cap\,{}^\theta F'_\gamma\mathscr O$. In particular,
$\mathbf B$ induces bases in $\gr'_\bullet\mathscr O$,
$\mathscr O_{[\leq0]}$ and $\mathscr O_{[>0]}$. Then
$\overline\Delta'$ should map this basis of $\gr'_\bullet\mathscr O$
onto the tensor product basis of
$\mathscr O_{[\leq0]}\otimes\mathscr O_{[>0]}$.

A similar clarification is in order for condition (iv). However, one
can show that conditions (iii) and (iv) are in fact equivalent, so
including them both is redundant.

\begin{other}{Example}
The quantized enveloping algebra $U_q(\mathfrak n)$ can be categorified
by modules over the Khovanov--Lauda--Rouquier algebras. In this context,
Varagnolo and Vasserot showed that for $G$ simply laced, the elements
of the canonical basis of $U_q(\mathfrak n)$ (up to a power of $q$)
correspond to the indecomposable projective graded objects
\cite{VaragnoloVasserot}. From Tingley and Webster's results
in~\cite{TingleyWebster}, we then deduce that the dual canonical
basis is a polite basis of $\mathscr O$. Now choose
$\theta\in Q_{\mathbb R}^*$ and adopt the notation of
section~\ref{ss:StabilityON}. Then the subspaces
${}^\theta F'_0\mathscr O=\mathscr O_{[\leq0]}=\mathscr O^{N_{>0}}$ and
${}^\theta F''_0\mathscr O=\mathscr O_{[\geq0]}={}^{N_{<0}}\mathscr O$
are spanned by their intersection with the dual canonical basis.
This proves afresh a result of Kimura~\cite{Kimura}. (Kimura's
actual result is in truth deeper, as it is valid for quantized
symmetrizable Kac--Moody algebra.)
\end{other}

Given a convex polytope $K$ contained in $Q_{\mathbb R}$, we define
$\mathscr S(K)$ as the vector subspace of $\mathscr O$ spanned by
weight functions $f$ such that $\underline\Pol(f)\subset K$.

\begin{proposition}
\label{pr:ConvComp}
Let $\mathbf B$ be a polite basis of $\mathscr O$. Then for each convex
polytope $K$ in $Q_{\mathbb R}$ the subspace $\mathscr S(K)$ is spanned
by $\mathbf B\cap\mathscr S(K)$.
\end{proposition}

\begin{proof}
We may without loss of generality restrict our considerations to a
weight subspace $\mathscr O_\lambda$ with $\lambda\in Q_-$. We set
$M=Q_-\cap(\lambda-Q_-)$. The Hahn--Banach theorem implies the
existence of finitely many elements $\theta_1$, \dots, $\theta_n$ in
$Q_{\mathbb R}^*$ and of degrees $\gamma_1\in\Gamma_{\theta_1}$, \dots,
$\gamma_n\in\Gamma_{\theta_n}$ such that, for each $\mu\in M$,
$$\mu\in K\Leftrightarrow\Bigl(\forall i\in\inter1n,\
J_{\theta_i}(\mu)\in\gamma_i+\Pi'\Bigr).$$
As $\underline L(f)\subset M$ for any function $f\in\mathscr O_\lambda$,
we obtain that
$$\mathscr O_\lambda\cap\mathscr S(K)=\{f\in\mathscr O_\lambda\mid
\underline L(f)\subset K\}=\mathscr O_\lambda\cap\bigcap_{i=1}^n
\Bigl({}^{\theta_i}F'_{\gamma_i}\mathscr O\Bigr).$$
The proposition now follows from the compatibility of $\mathbf B$ with
all the subspaces occurring in the right-hand side.
\end{proof}

\begin{corollary}
\label{co:SingleMaximal}
Let $\mathbf B$ be a polite basis of $\mathscr O$, let $f\in\mathbf B$,
and let $\theta\in Q_{\mathbb R}^*$. Expand $f$ as a sum
$\sum_{i=1}^\ell a_i$ of ordered monomials with respect to the
$\Gamma_\theta$-grading of $\mathscr O$. Then there exists a term $a_i$
such that $\Pol^\wedge_\theta(f)=\Pol^\wedge(a_i)$.
\end{corollary}

\begin{proof}
We may assume without loss of generality that the functions $a_i$
are homogeneous with respect to the weight grading on $\mathscr O$.
From Proposition~\ref{pr:Expansion}, we know that
$\Pol^\wedge_\theta(a_i)\subset\Pol^\wedge_\theta(f)$ for each
$i\in\inter1\ell$. On the other hand, $a_i$ belongs to
$\mathscr S(\underline\Pol(a_i))$, which is spanned by its
intersection with $\mathbf B$. So each $a_i$, and therefore
$f$, can be written as a linear combination of elements in
$\bigcup_{i=1}^\ell(\mathbf B\cap\mathscr S(\underline\Pol(a_i)))$.
As $\mathbf B$ is a basis, there necessarily exists an index
$i\in\inter1\ell$ such that $f\in\mathscr S(\underline\Pol(a_i))$.
We obtain $\Pol_\theta(f)\subset\Pol_\theta(a_i)$, then
$\Pol^\wedge_\theta(f)=\Pol^\wedge_\theta(a_i)$.
\end{proof}

The root hyperplanes draw a fan in the vector space $Q_{\mathbb R}^*$,
called the Weyl fan. We say that a convex polytope $P$ contained in
$Q_{\mathbb R}$ is GGMS if its normal fan is a coarsening of the Weyl
fan (GGMS stands for Gelfand, Goresky, MacPherson and Serganova;
see~\cite{Kamnitzer10}, section~2.4).

\begin{corollary}
\label{co:GGMS}
Let $\mathbf B$ be a polite basis of $\mathscr O$ and let $f\in\mathbf B$.
Then the polytope $\underline\Pol(f)$ is GGMS.
\end{corollary}

\begin{proof}
Let $-\theta$ be an element in a face of the normal fan to
$\underline\Pol(f)$, so that the locus where $\theta$ achieves its
minimum on $\underline\Pol(f)$ is a face of positive dimension.
The upper rim of $\Pol_\theta(f)$ has then a nontrivial segment of
slope $0$. Applying Corollary~\ref{co:SingleMaximal}, we conclude
that $\mathscr O_{[0]}$ is not reduced to the base field $k$.
According to the discussion in section~\ref{ss:StabilityON}, there
exists a positive root $\alpha$ such that $\theta(\alpha)=0$, so
$\theta$ belongs to a face of the Weyl fan. To sum up, each chamber
of the Weyl fan is contained in a chamber of the normal fan to
$\underline\Pol(f)$.
\end{proof}

\begin{proposition}
\label{pr:UpperRim}
Let $\mathbf B$ be a polite basis of $\mathscr O$ and let $\theta\in
Q_{\mathbb R}^*$ be such that the slopes $\theta(\alpha)/\height\alpha$
of the positive roots $\alpha$ are pairwise different. Then the map
$f\mapsto\Pol^\wedge_\theta(f)$ is injective on $\mathbf B$.
\end{proposition}

\begin{proof}
Let $f_1$ and $f_2$ be two elements of $\mathbf B$ that have the same
image by the map $\Pol^\wedge_\theta$. By our assumption on $\theta$,
there is a unique ordered monomial $a$ relative to the
$\Gamma_\theta$-grading of $\mathscr O$ such that
$\Pol^\wedge(a)=\Pol^\wedge_\theta(f_1)=\Pol^\wedge_\theta(f_2)$.
By Corollary~\ref{co:SingleMaximal}, $a$ occurs in the expansions of
both $f_1$ and $f_2$ as sums of ordered monomials, and there exists
a scalar $t$ such that $f_1+tf_2$ is a sum $\sum_{i=1}^\ell b_i$ of
ordered monomials, homogeneous with respect to the weight grading on
$\mathscr O$, with $\Pol^\wedge(b_i)\subsetneq\Pol^\wedge(a)$ for all
$i\in\inter1\ell$. Certainly then $f_1+tf_2$ is a linear combination
of elements in
$\bigcup_{i=1}^\ell(\mathbf B\cap\mathscr S(\underline\Pol(b_i)))$, but
as neither $f_1$ nor $f_2$ belongs to this union, we necessarily have
$f_1+tf_2=0$, in particular~$f_1=f_2$.
\end{proof}

\begin{corollary}
\label{co:PolCarac}
Let $\mathbf B$ be a polite basis of $\mathscr O$. Then the map
$f\mapsto\underline\Pol(f)$ is injective on $\mathbf B$.
\end{corollary}

The map $K\mapsto\mathscr S(K)$ is a filtration of $\mathscr O$ indexed
by the set of all convex polytopes, endowed with the inclusion order.
We denote the associated graded by $\gr_\bullet\mathscr O$.
By Proposition~\ref{pr:ConvComp}, any polite basis $\mathbf B$ of
$\mathscr O$ is compatible with this filtration and thus allows
to compute this associated graded. Explicitly, we get that
$\gr_K\mathscr O$ is one dimensional if $K$ is of the form
$\underline\Pol(f)$ for $f\in\mathbf B$ and is zero otherwise.

Accordingly, the set $\underline\Pol(\mathbf B)$ is the same for all
polite bases $\mathbf B$. As mentioned in the introduction, the elements
of this set are the Mirković--Vilonen polytopes~\cite{Anderson}.
Kamnitzer proved that among the GGMS polytopes, these are characterized
by the shape of their $2$-faces, which is constrained by the tropical
Plücker relations~\cite{Kamnitzer10}. A more general statement is that
faces of Mirković--Vilonen polytopes are Mirković--Vilonen polytopes
of smaller ranks~\cite{BaumannKamnitzerTingley}.

\begin{other}{Remark}
This last fact can be proved in our set-up in the following way.
As in section~\ref{ss:StabilityON}, the datum of an element
$\theta\in Q_{\mathbb R}^*$ defines subgroups $N_{>0}$, $N_0$ and $N_{<0}$
of $N$, and the space $\mathscr O_{[0]}={}^{N_{<0}}\mathscr O^{N_{>0}}$ is
naturally isomorphic to the algebra of regular functions on $N_0$.
As $N_0$ can be regarded as the maximal unipotent subgroup of a reductive
group with root system $\Phi\cap\ker\theta$, it makes sense to speak of
polite bases of $\mathscr O_{[0]}$. (The only true matter of concern is
the normalization condition (ii), which should be adequately managed.)
And as a matter of fact, one easily shows that if $\mathbf B$ is a
polite basis of $\mathscr O$, then $\mathbf B\cap\mathscr O_{[0]}$ is a
polite basis of $\mathscr O(N_0)$. Now consider such a basis $\mathbf B$
and pick $f\in\mathbf B$. Any face of $\underline\Pol(f)$ can be obtained
as the minimum locus on $\underline\Pol(f)$ of an element
$\theta\in Q_{\mathbb R}^*$. Using this $\theta$ as a stability parameter,
the upper ridge of $\Pol_\theta(f)$ contains a segment of slope $0$,
with endpoints $\alpha$ and $\beta$. Denoting by $\bar f$ the image of
$f$ in $\gr'_\alpha\gr''_\beta\mathscr O$ and adopting the notation of
Proposition~\ref{pr:Slicing}, we can write
$\overline\Delta_2(\bar f)=f_+\otimes f_0\otimes f_-$ with
$f_+$, $f_0$ and $f_-$ in~$\mathbf B$. Then $f_0$ belongs to
$\mathbf B\cap\mathscr O_{[0]}$, and the face of $\underline\Pol(f)$
that we considered is a translate of the polytope $\underline\Pol(f_0)$.
\end{other}

\begin{proposition}
\label{pr:TransitionMat}
Let $\mathbf B_1$ and $\mathbf B_2$ be polite bases of $\mathscr O$.
We index both of them by the Mirković--Vilonen polytopes by means of
the map $\underline\Pol$. Then the transition matrix between $\mathbf B_1$
and $\mathbf B_2$ is upper unitriangular with respect to the inclusion
order of the polytopes.
\end{proposition}

\begin{proof}
The triangularity of the transition matrix is a consequence of the
compatibility of both bases with the filtration $K\mapsto\mathscr S(K)$.
The fastest way to prove that there are only ones on the diagonal is to
observe that polite bases are always perfect (see section
\ref{ss:PerfectBases} for the specific definition we use in this paper)
and to invoke the fact that the transition matrix between two perfect
bases is unitriangular (see for instance Proposition~2.6 in \cite{Baumann}
for the dual version).
\end{proof}

\subsection{The dual semicanonical basis}
\label{ss:SCBasis}
In this section we assume that $G$ is simply laced and we show that the
dual semicanonical basis of $\mathscr O$ is polite. We start with a few
recollections to settle the notation and refer to the literature for
complete explanations \cite{GeissLeclercSchroer05,Lusztig00}.

Let $\Lambda$ be the proprejective algebra on the Dynkin diagram of $G$
over the field of complex numbers. This is the path algebra of a quiver
with relations, with vertex set $I$. We regard the dimension-vector
$\dimvec M$ of a $\Lambda$-module $M$ as an element of $Q_+$, identified
to $\mathbb N^I$ by means of the basis $\{\alpha_i\}_{i\in I}$.
The Harder--Narasimhan polytope of a $\Lambda$-module $M$ is the convex
hull in the space $Q_{\mathbb R}$ of the set of dimension-vectors of all
the submodules of $M$.

We define the weight of a word $\mathbf i=i_1\cdots i_n$ in the
alphabet $I$ as $\alpha_{i_1}+\cdots+\alpha_{i_n}$. To such a word,
we associate the monomial $e_{\mathbf i}=e_{i_1}\cdots e_{i_n}$ in
the Chevalley generators of $\mathfrak n$. If $M$ is a $\Lambda$-module
of dimension-vector $\nu$ and $\mathbf i$ is a word of weight $\nu$,
we denote by $\mathscr F_{\mathbf i,M}$ the set of all flags
$0=N_0\subset N_1\subset\dots\subset N_n=M$ of $\Lambda$-submodules of
$M$ such that $\dimvec(N_j/N_{j-1})=\alpha_{i_j}$ for all $j\in\inter1n$.
With this notation, each $\Lambda$-module $M$ of dimension-vector $\nu$
determines a function $\varphi_M\in\mathscr O_{-\nu}$ such that
$\langle e_{\mathbf i},\varphi_M\rangle$ is the Euler characteristic of
$\mathscr F_{\mathbf i,M}$ for all words $\mathbf i$ of weight $\nu$.

\begin{proposition}
\label{pr:HNPolytope}
The Harder--Narasimhan polytope of a $\Lambda$-module $M$ is equal to
$-\underline\Pol(\varphi_M)$.
\end{proposition}

\begin{proof}
We set $\nu=\dimvec M$ and pick $\mu\in(-\underline L(\varphi_M))$. The
component of weight $(-\mu,-\nu+\mu)$ of $\Delta(\varphi_M)$ is nonzero,
so there exist words $\mathbf i$ and $\mathbf j$ of weights $\mu$ and
$\nu-\mu$, respectively, such that
$\langle e_{\mathbf i}\otimes e_{\mathbf j},\Delta(\varphi_M)\rangle\neq0$.
Denoting by $\mathbf{ij}$ the concatenation of these words, we get
$\langle e_{\mathbf{ij}},\varphi_M\rangle\neq0$, so
$\mathscr F_{\mathbf{ij},M}\neq\varnothing$ and $M$
has a submodule of dimension-vector $\mu$. We conclude that
$-\underline L(\varphi_M)$, and therefore its convex hull
$-\underline\Pol(\varphi_M)$, is contained in the Harder--Narasimhan
polytope of $M$.

To prove the reverse inclusion, we choose a vertex $\mu$ of the
Harder--Narasimhan polytope of $M$. Then $M$ contains a unique submodule
$N$ of dimension-vector $\mu$ (\cite{BaumannKamnitzerTingley},
section~3.2). If $\mathbf i$ and $\mathbf j$ are words of weights
$\mu$ and $\nu-\mu$, respectively, then any flag in
$\mathscr F_{\mathbf{ij},M}$ contains $N$, whence an isomorphism
$\mathscr F_{\mathbf{ij},M}\cong\mathscr F_{\mathbf i,N}\times
\mathscr F_{\mathbf j,M/N}$. Taking Euler characteristics, we get
$\langle e_{\mathbf{ij}},\varphi_M\rangle=\langle e_{\mathbf i},
\varphi_N\rangle\langle e_{\mathbf j},\varphi_{M/N}\rangle$. Since
this holds for all such words $\mathbf i$ and $\mathbf j$, we deduce
that the component of weight $(-\mu,-\nu+\mu)$ of $\Delta(\varphi_M)$
is equal to $\varphi_N\otimes\varphi_{M/N}$, hence is not zero
(\cite{GeissLeclercSchroer11}, Lemma~9.5). Therefore
$-\mu\in\underline L(\varphi_M)$. We conclude that all the vertices of the
Harder--Narasimhan polytope of $M$ belong to $-\underline\Pol(\varphi_M)$.
\end{proof}

Given $\nu\in Q_+$, we denote the representation space of
$\Lambda$-modules (aka Lusztig's nilpotent variety) in vector-dimension
$\nu$ by $\rep(\Lambda,\nu)$. If $M$ is the $\Lambda$-module defined by
the point $m\in\rep(\Lambda,\nu)$, we allow ourselves to write $\varphi_m$
instead of $\varphi_M$. The map $m\mapsto\varphi_m$ from $\rep(\Lambda,\nu)$
to $\mathscr O_{-\nu}$ is constructible. To any irreducible component
$Y$ of $\rep(\Lambda,\nu)$, one can therefore associate a function
$\rho_Y\in\mathscr O_{-\nu}$ so that $\{m\in Y\mid\varphi_m=\rho_Y\}$
contains a dense open subset of $Y$. The dual semicanonical basis of
$\mathscr O$ is the family of all these functions $\rho_Y$.

Each $\theta\in Q_{\mathbb R}^*$ defines a torsion pair
$(\mathscr I_\theta,\overline{\mathscr P}_\theta)$ in $\Lambda$-mod
(\cite{BaumannKamnitzerTingley}, section~3.1). For each $\Lambda$-module
$M$, denote by $M_t$ and $M_f=M/M_t$ the torsion and locally-free part
of $M$ with respect to this torsion pair. The map $M\mapsto(M_t,M_f)$
extends to irreducible components: given a dimension-vector $\nu$ and an
irreducible component $Y$ of $\rep(\Lambda,\nu)$, there exists a
dimension-vector $\mu$ and irreducible components $Y_t$ and $Y_f$ of
$\rep(\Lambda,\mu)$ and $\rep(\Lambda,\nu-\mu)$, respectively, such that
the general point of $Y_t\times Y_f$ is of the form $(M_t,M_f)$ with $M$
a general point in $Y$. Moreover $\mu$, $Y_t$ and $Y_f$ are uniquely
defined and the map $Y\mapsto(Y_t,Y_f)$ is injective (\textit{loc.~cit.},
Theorem~4.4). From Proposition~\ref{pr:HNPolytope} and its proof, we
deduce that $-J_\theta(\mu)$ is the largest element in $L_\theta(\rho_Y)$
with respect to the total order $\leq'$ and that the component of weight
$(-\mu,-\nu+\mu)$ of $\Delta(\rho_Y)$ is~$\rho_{Y_t}\otimes\rho_{Y_f}$.

\begin{proposition}
\label{pr:SCPolite}
The dual semicanonical basis of $\mathscr O$ is polite.
\end{proposition}

\begin{proof}
By construction, the dual semicanonical basis satisfies the
condition~(i) in the definition of polite bases. For condition~(ii),
we choose $(i,n)\in I\times\mathbb N$, denote the unique $\Lambda$-module
of dimension-vector $n\alpha_i$ by $nS_i$, observe that
$\langle e_{ii\cdots i},\varphi_{nS_i}\rangle$ is equal to $n!$,
the Euler characteristic of the complete flag variety of an
$n$-dimensional complex vector space, and conclude that
$\zeta_i^n$ is equal to $\varphi_{nS_i}$, so belongs to the dual
semicanonical basis.

We now show that the dual semicanonical basis satisfies condition~(iii).
Let $\theta\in Q_{\mathbb R}^*$, let $(\alpha,\gamma)\in\Gamma_\theta^2$,
and let $f$ be a homogenous element in ${}^\theta F'_\alpha\mathscr O$
of degree $\gamma$. Let us expand $f$ on the dual semicanonical basis
and write $f=\sum_{i=1}^\ell c_i\rho_{Y_i}$, where each scalar
$c_i$ is nonzero and where $Y_i$ is an irreducible component of a
representation space $\rep(\Lambda,\nu_i)$. We decompose each $Y_i$
according to the torsion pair
$(\mathscr I_\theta,\overline{\mathscr P}_\theta)$, producing a
dimension-vector $\mu_i$ and irreducible components $(Y_i)_t$ and
$(Y_i)_f$ of $\rep(\Lambda,\mu_i)$ and $\rep(\Lambda,\nu_i-\mu_i)$,
respectively. Let $\beta$ be the largest element among the degrees
$-J_\theta(\mu_i)$ with respect to the total order $\leq'$ and let
$I$ be the set of all indices $i$ such that $-J_\theta(\mu_i)=\beta$.
Then the homogeneous component of degree
$(\beta,\gamma-\beta)$ of $\Delta(f)$ is
$\sum_{i\in I}c_i\,\rho_{(Y_i)_t}\otimes\rho_{(Y_i)_f}$. Since the
pairs $((Y_i)_t,(Y_i)_f)$ are all different, this sum cannot be zero,
so $\beta\in L(f)$ and therefore $\beta\leq'\alpha$. It follows that
for each $i\in\inter1\ell$, the set $L_\theta(\rho_{Y_i})$ is
contained in $\alpha+\Pi'$ and the function $\rho_{Y_i}$
belongs to ${}^\theta F'_\alpha\mathscr O$. We conclude that
${}^\theta F'_\alpha\mathscr O$ is spanned by its intersection with
the dual semicanonical~basis. The same arguments imply the
compatibility of the dual semicanonical basis with the isomorphism
$\overline\Delta'$.

Condition~(iv) in the definition of polite bases can be checked
similarly, using instead the torsion pair
$(\overline{\mathscr I}_\theta,\mathscr P_\theta)$ in $\Lambda$-mod.
\end{proof}

\subsection{Comparison with perfect bases}
\label{ss:PerfectBases}
Berenstein and Kazhdan introduced in \cite{BerensteinKazhdan} the notion
of perfect bases for locally finite $U(\mathfrak n)$-modules. The
following is a reformulation of their definition in the case of the
$U(\mathfrak n)$-module $\mathscr O$, actually a strengthened version
of it where the values of some structure constants is prescribed.
For $i\in I$ and $f\in\mathscr O$, let $\ell_i(f)$ denote the smallest
nonnegative integer $n$ such that $e_i^{n+1}\cdot f=0$. Then a basis
$\mathbf B$ of $\mathscr O$ consisting of weight vectors is said to be
perfect if, for each $(i,n)\in I\times\mathbb N$, the rule
$b\mapsto e_i^{(n)}\cdot b$ defines an injective map
$\{b\in\mathbf B\mid\ell_i(b)=n\}\to\{b\in\mathbf B\mid e_i\cdot b=0\}$.

We will say that a basis $\mathbf B$ of $\mathscr O$ is biperfect if
both $\mathbf B$ and its image by the involution $*$ are perfect.

\begin{proposition}
\label{pr:Perfectness}
A polite basis of $\mathscr O$ is biperfect.
\end{proposition}

\begin{proof}
We fix $i\in I$ and choose $\theta\in Q_{\mathbb R}^*$ so that
$\theta(\alpha_i)<0$ and that $\theta$ takes positive values on all
the other positive roots. With this stability parameter, the group
$N_{<0}$ defined in section~\ref{ss:StabilityON} is the additive
group $\mathbb G_a$ defined by the root $\alpha_i$ and the algebra
$\mathscr O_{[<0]}$ is the polynomial algebra~$k[\zeta_i]$. Also,
$\mathscr O_{[\geq0]}={}^{N_{<0}}\mathscr O=\{f\in\mathscr O\mid
e_i\cdot f=0\}$.

We set $\gamma=J_\theta(-\alpha_i)$; this is the degree of $\zeta_i$
relative to the $\Gamma_\theta$-grading on $\mathscr O$. From
Proposition~\ref{pr:Splitting''}, we deduce that $\gr''_\bullet\mathscr O$
is concentrated in degrees $n\gamma$ with $n\in\mathbb N$.

If a homogeneous element $f\in\mathscr O$ is in $F''_{n\gamma}\mathscr O$,
then $(n+1)\gamma\notin R(f)$ and therefore $e_i^{n+1}\cdot f=0$.
Conversely, if $f$ is not in $F''_{n\gamma}\mathscr O$, then there exists
$\beta\in R(f)$ such that $\beta\notin n\gamma+\Pi''$. In other words,
there exists an element $x\in U(\mathfrak n)$ of degree $\beta$ (for the
$\Gamma_\theta$-grading of $U(\mathfrak n)$) such that $x\cdot f\neq0$,
and without loss of generality we can assume that $x$ is a PBW monomial
ending with some power of $e_i$. The fact that $\beta\notin n\gamma+\Pi''$
forces this power to be at least $e_i^{n+1}$, and we conclude that
$e_i^{n+1}\cdot f\neq0$.

The previous paragraph shows that for a homogeneous element $f\in\mathscr O$,
the smallest nonnegative integer $n$ such that $f\in F''_{n\gamma}\mathscr O$
is $\ell_i(f)$. For this integer $n=\ell_i(f)$, denoting by $\bar f$ the
image of $f$ in $\gr''_{n\gamma}\mathscr O$, we then have
$\overline\Delta''\bigl(\bar f\bigr)=e_i^{(n)}\cdot f\otimes\varphi_i^n$,
because $e_i^{(n)}$ spans the homogeneous component of degree $n\gamma$
in $U(\mathfrak n)$ and $\varphi_i^n$ is the dual basis element in
$\mathscr O$.

Now let $\mathbf B$ be a polite basis of $\mathscr O$. For $n\in\mathbb N$,
set $\mathbf B_n=\{b\in\mathbf B\mid\ell_i(b)=n\}$. In view of our previous
observations, the axioms of a polite basis require that $\mathbf B_0$
is a basis of $F''_0\mathscr O=\mathscr O_{[\geq0]}$, that
$\bigl\{\bar b\bigm|b\in\mathbf B_n\bigr\}$ is a basis of
$\gr''_{n\gamma}\mathscr O$, and that $\bar b\mapsto e_i^{(n)}\cdot b$
maps bijectively this basis onto~$\mathbf B_0$. Therefore $\mathbf B$
is a perfect basis of~$\mathscr O$.

Employing $-\theta$ as stability parameter and arguing in similar fashion,
we deduce from the compatibility of $\mathbf B$ with the filtration
$F'_\bullet$ and with the isomorphism $\overline\Delta'$ that $\mathbf B$
enjoys the same property with respect to the right action of $e_i$ on
$\mathscr O$. Therefore $\{b^*\mid b\in\mathbf B\}$ is also perfect.
\end{proof}

Perfect bases are automatically endowed with a structure of crystal in
the sense of Kashiwara. We will not recall the definition of this notion;
let us just say that we will use the standard notation, with crystal
operators denoted by $\tilde e_i$ and $\tilde f_i$ and with functions
$\varepsilon_i$ and $\varphi_i$ for each $i\in I$. One important result
from \cite{BerensteinKazhdan} is that the crystals of the perfect bases
of $\mathscr O$ are all isomorphic. The abstract crystal they share is
denoted by $B(\infty)$.

Proposition~\ref{pr:Perfectness} therefore entails that any polite basis
$\mathbf B$ is indexed by the crystal $B(\infty)$. Also, as we saw in
section~\ref{ss:PoliteBases}, the map $\underline\Pol$ defines a bijection
from $\mathbf B$ onto the set of Mirković--Vilonen polytopes.
We thus obtain a bijection from $B(\infty)$ onto the set of
Mirković--Vilonen polytopes, and this bijection does not depend on
$\mathbf B$ (\cite{Baumann}, Lemma~2.1). It follows that the set of
Mirković--Vilonen polytopes has a natural crystal structure, isomorphic
to $B(\infty)$. This result goes back to Kamnitzer, who described
this bijection in the most explicit manner~\cite{Kamnitzer07}.

\begin{other}{Remark}
The converse of Proposition~\ref{pr:Perfectness} is not true: biperfect
bases are not necessarily polite. Suppose indeed given a pair $(b',b'')$
of elements in $B(\infty)$ which have the same weight and satisfy
$\varepsilon_i(b')<\varepsilon_i(b'')$ and
$\varepsilon_i^*(b')<\varepsilon_i^*(b'')$ for each $i\in I$,
but also $\underline\Pol(b')\not\subset\underline\Pol(b'')$.
By Proposition~\ref{pr:TransitionMat}, the transition matrix between any
two polite bases of $\mathscr O$ have a zero entry at position $(b'',b')$.
However, Fei explains in sect.~10.1 of \cite{Fei} that one can find two
biperfect bases of $\mathscr O(N)$ such that the transition matrix between
these bases has a nonzero entry at that position. In this case certainly
one of these biperfect bases is not polite.

The main point now is to show that such pairs $(b',b'')$ do exist.
The following is Example~2.7~(ii) in \cite{Baumann} (a smaller example
is given in \cite{Fei}, Example~10.7). Here we are in type $A_4$ with
the standard numbering of the vertices of the Dynkin diagram.
We consider the two elements of~$B(\infty)$
$$b'=\Bigl(\tilde f_3\tilde f_2\tilde f_1\tilde f_4\tilde f_3
\tilde f_2\Bigr)^4(1)\quad\text{and}\quad
b''=\Bigl(\tilde f_1^2\tilde f_3^5\tilde f_2^5\tilde f_4^2\Bigr)
\Bigl(\tilde f_2\tilde f_3\tilde f_4\tilde f_1\tilde f_2\tilde f_3\Bigr)
\Bigl(\tilde f_4\tilde f_3\tilde f_2\tilde f_1\Bigr)(1).$$
They have the same weight $\nu$, and in the preprojective model
\cite{KashiwaraSaito,Lusztig90}, they correspond to the irreducible
components of $\rep(\Lambda,\nu)$ whose general points are the
$\Lambda$-modules
$$M'=\biggl(\begin{smallmatrix}&&3&\\&2&&4\\1&&3&\\&2&&
\end{smallmatrix}\biggr)^{\oplus4}\quad\text{and}\quad
M''=\begin{smallmatrix}&&&4\\&&3&\\&2&&\\1&&&\end{smallmatrix}\oplus
\begin{smallmatrix}&2&&\\1&&3&\\&2&&4\\&&3&\end{smallmatrix}\oplus
\Bigl(\begin{smallmatrix}1&&3&\\&2&&4\end{smallmatrix}\Bigr)^{\oplus2}
\oplus\Bigl(\begin{smallmatrix}&3\\2&\end{smallmatrix}\Bigr)^{\oplus3},$$
respectively. Looking at the heads and the socles of these modules, we
find that
\begin{xalignat*}2
(\varepsilon_1(b'),\varepsilon_2(b'),\varepsilon_3(b'),
\varepsilon_4(b')&=(0,0,4,0),&
(\varepsilon_1(b''),\varepsilon_2(b''),\varepsilon_3(b''),
\varepsilon_4(b'')&=(2,1,5,1),\\[3pt]
(\varepsilon_1^*(b'),\varepsilon_2^*(b'),\varepsilon_3^*(b'),
\varepsilon_4^*(b'))&=(0,4,0,0),&
(\varepsilon_1^*(b''),\varepsilon_2^*(b''),\varepsilon_3^*(b''),
\varepsilon_4^*(b''))&=(1,5,1,2).
\end{xalignat*}
On the other hand, adopting the notation of \cite{BaumannKamnitzer}, we get
$$N(s_2s_4s_3\,\omega_3)=\begin{smallmatrix}1&&3\\&2&\end{smallmatrix}$$
and we can then compute
$$\dim\Hom_\Lambda(N(s_2s_4s_3\,\omega_3),M')=4>2=
\dim\Hom_\Lambda(N(s_2s_4s_3\,\omega_3),M'').$$
Using Theorem~6.3 in \textit{loc.\ cit.} we conclude that
$\underline\Pol(b')\not\subset\underline\Pol(b'')$.

Computer experiments lead us to believe that in this type $A_4$, the algebra
$\mathscr O$ has several biperfect bases but only one polite basis.
\end{other}

\section{Mirković and Vilonen's basis}
\label{se:MVBasis}

\subsection{Recollection about the Geometric Satake Equivalence}
\label{ss:GeometricSatake}
In this section, we present a very brief summary of the Geometric
Satake Equivalence. We direct the reader to~\cite{MirkovicVilonen}
for additional details.

We carry on with the notation set up in section~\ref{ss:Notation} and
denote the Borel subgroup opposite to $B$ with respect to $T$ by $B_-$.
Since the group $T$ is the quotient of $B_-$ by its unipotent radical,
any character $\lambda$ of $T$ can be inflated to a linear character of
$B_-$, still denoted by $\lambda$. One can then consider the coinduced
$G$-module (customarily called costandard)
$$\nabla(\lambda)=\bigl\{f\in\mathscr O(G)\bigm|
\forall(b,g)\in B_-\times G,\;f(bg)=\lambda(b)f(g)\bigr\}.$$
If $\lambda$ is dominant, then $\nabla(\lambda)$ has $\lambda$ for
highest weight and its character is given by Weyl's formula.
The Geometric Satake Equivalence realizes $\nabla(\lambda)$ as the
homology of a certain perverse sheaf $\mathcal I_*(\lambda,k)$ on the
affine Grassmannian of the Langlands dual of $G$.

Let us set up the relevant notation. The Langlands dual of $G$ will be
denoted by $G^\vee$ (for us, $G^\vee$ will be a complex algebraic group).
Its maximal torus $T^\vee$ has $P$ for cocharacter lattice and the root
system of $G^\vee$ is the coroot system of $G$. We choose an additive
one-parameter subgroup $x_{\alpha^\vee}:\mathbb G_a\to G^\vee$
for each coroot $\alpha^\vee$. We denote by $N^\vee_\pm$ the subgroups
of $G^\vee$ generated by the subgroups $x_{\alpha^\vee}$ for
$\alpha\in\Phi_\pm$. We introduce the ring $\mathbb O=\mathbb C[\![z]\!]$
and its fraction field $\mathbb K=\mathbb C(\!(z)\!)$.

The affine Grassmannian of $G^\vee$ is the homogeneous space
$\Gr=G^\vee(\mathbb K)/G^\vee(\mathbb O)$. It is the set of
$\mathbb C$-points of a reduced projective ind-scheme over $\mathbb C$.
Each weight $\lambda\in P$ can be regarded as a point $z^\lambda$ in
$T^\vee(\mathbb K)$, hence defines a point $L_\lambda$ in $\Gr$. The orbit
through $L_\lambda$ under the action of $N^\vee_+(\mathbb K)$ (respectively,
$N^\vee_-(\mathbb K)$) is denoted by $S_\lambda$ (respectively,
$T_\lambda$). Then
$$\Gr=\bigsqcup_{\lambda\in P}S_\lambda=\bigsqcup_{\lambda\in P}T_\lambda,
\qquad\overline{S_\lambda}=\bigsqcup_{\mu\in Q_+}S_{\lambda-\mu},
\qquad\overline{T_\lambda}=\bigsqcup_{\mu\in Q_+}T_{\lambda+\mu},$$
and for each $\mu\in P$, the action of $z^\mu$ on $\Gr$ sends $S_\lambda$
onto $S_{\lambda+\mu}$ and $T_\lambda$ onto $T_{\lambda+\mu}$.

Let $\lambda$ be a dominant weight. We denote the orbit through $L_\lambda$
under the action of $G^\vee(\mathbb O)$ by $\Gr^\lambda$ and consider the
perverse sheaf
$$\mathcal I_*(\lambda,k)={}^p\tau_{\leq0}\;(j_\lambda)_*\;
\underline k_{\Gr^\lambda}[2\rho(\lambda)]$$
in the derived category $D^b(\Gr,k)$ of constructible sheaves on
$\Gr$. Here $2\rho:P\to\mathbb Z$ is the sum of the positive coroots,
$j_\lambda:\Gr^\lambda\to\Gr$ is the inclusion map, $(j_\lambda)_*$
is the (derived) direct image, and ${}^p\tau_{\leq0}$ is the
truncation functor for the perverse $t$-structure. Then, under the
Geometric Satake Equivalence, the module $\nabla(\lambda)$ is the
hypercohomology $\text H^\bullet(\Gr,\mathcal I_*(\lambda,k))$,
and for all $\nu\in P$, the weight space $\nabla(\lambda)_\nu$ is
$\text H^{2\rho(\nu)}(T_\nu,t_\nu^!\,\mathcal I_*(\lambda,k))$, where
$t_\nu:T_\nu\to\Gr$ is the inclusion map. Rewriting the sheaf as
$$\mathcal I_*(\lambda,k)={}^p\tau_{\leq0}\;(j_\lambda)_*\;
\mathbb D_{\Gr^\lambda}[-2\rho(\lambda)],$$
we get (see \cite{MirkovicVilonen}, Proposition~3.10)
$$\nabla(\lambda)_\nu\cong\text H^{2\rho(\nu-\lambda)}
\bigl(\Gr^\lambda\cap T_\nu,\mathbb D_{\Gr^\lambda\cap T_\nu}\bigr)
\cong\text H_{2\rho(\lambda-\nu)}^\BM\bigl(\Gr^\lambda\cap T_\nu,k\bigr).$$
Furthermore, all the irreducible components of $\Gr^\lambda\cap T_\nu$
have dimension $\rho(\lambda-\nu)$, so their fundamental classes form
a basis of $\nabla(\lambda)_\nu$.

Let $q:P\to\mathbb Q$ be a positive definite quadratic form on $P$,
invariant under the Weyl group. It determines a central extension
$\widetilde{\mathfrak g^\vee}$ of the Lie algebra
$\mathfrak g^\vee\otimes\mathbb C[z,z^{-1}]$ by~$\mathbb C$.
The basic representation\footnote{If $G$ is simple and if one chooses
$q$ such that $q(\alpha)=1$ for each short root $\alpha$, then
$\widetilde{\mathfrak g^\vee}$ is an affine untwisted Kac--Moody
algebra---up to the derivation $d/dz$---and the basic representation
is the integrable representation with highest weight $\Lambda_0$.}
$V$ of $\widetilde{\mathfrak g^\vee}$ provides an embedding
$\Upsilon:\Gr\to\mathbb P(V)$ and a $G^\vee(\mathbb K)$-equivariant
line bundle $\mathscr L=\Upsilon^*\mathscr O(1)$ on $\Gr$. For each
dominant weight $\lambda$, the cup-product with $c_1(\mathscr L)$
defines an endomorphism of the vector space
$\text H^\bullet(\Gr,\mathcal I_*(\lambda,k))$. Then the Geometric Satake
Equivalence, suitably normalized, identifies this endomorphism with
the action on $\nabla(\lambda)$ of the principal nilpotent element
$\sum_{i\in I}q(\alpha_i)e_i$.

\subsection{Cutting Mirković--Vilonen cycles}
\label{ss:CuttingMVCycles}
For $(\lambda,\nu)\in P^2$, it is known that
$S_\lambda\cap T_\lambda=\{L_\lambda\}$, that the intersection
$S_\lambda\cap T_\nu$ is non-empty if and only if $\lambda-\nu\in Q_+$,
and that in this case this intersection has pure dimension
$\rho(\lambda-\nu)$. The irreducible components of the closure
$\overline{S_\lambda\cap T_\nu}$ are called Mirković--Vilonen cycles
of weight $(\lambda,\nu)$.  We shall denote the set they form by
$\mathscr Z_{\lambda,\nu}$.

We now fix a regular element $\theta\in Q_{\mathbb R}^*$ and define
$$\Phi_\theta=\bigl\{\alpha\in\Phi\bigm|\theta(\alpha)<0\bigr\}.$$
We define $\mathbb M=z^{-1}\mathbb C[z^{-1}]$, a subspace of
$\mathbb K$. For $\mathbb A\in\{\mathbb K,\mathbb O,\mathbb M\}$,
we denote by $U^\vee(\mathbb A)$ (respectively,
$U^\vee_\pm(\mathbb A)$) the subgroup of $G^\vee(\mathbb K)$
generated by the elements $x_{\alpha^\vee}(a)$ with $a\in\mathbb A$
and $\alpha\in\Phi_\theta$ (respectively,
$\alpha\in{}\Phi_\theta\cap\Phi_\pm$).

\begin{lemma}
\label{le:ZappaSzep}
\begin{enumerate}
\item
For $\mathbb A\in\{\mathbb K,\mathbb O,\mathbb M\}$, we have
$U^\vee(\mathbb A)=U^\vee_+(\mathbb A)
\bowtie U^\vee_-(\mathbb A)$.
\item
We have the following decompositions:
$$U^\vee(\mathbb K)=U^\vee(\mathbb M)
\bowtie U^\vee(\mathbb O)\quad\text{and}\quad
U^\vee_\pm(\mathbb K)=U^\vee_\pm(\mathbb M)
\bowtie U^\vee_\pm(\mathbb O).$$
\end{enumerate}
\end{lemma}

\begin{proof}
For all coroots $\alpha^\vee$ and $\beta^\vee$ that are not opposite,
there exist commutation relations of the form
$$x_{\alpha^\vee}(a)\,x_{\beta^\vee}(b)\,
x_{\alpha^\vee}(a)^{-1}\,x_{\beta^\vee}(b)^{-1}=
\prod_{i,j>0}x_{i\alpha^\vee+j\beta^\vee}\bigl(C_{i,j}a^ib^j\bigr)$$
with $C_{i,j}\in\mathbb C$; see for instance \cite{Carter}, chapter~5.
These relations allow to write any element in $U^\vee(\mathbb K)$
in a unique way as a product
\begin{equation*}
\tag{$*$}
\prod_{\alpha\in\Phi_\theta}x_{\alpha^\vee}(a_\alpha)
\end{equation*}
computed according to any given convex order on $\Phi_\theta$.
Choosing a convex order for which the elements in
$\Phi_\theta\cap\Phi_+$ are smaller than the elements in
$\Phi_\theta\cap\Phi_-$, this fact gives the decomposition
$U^\vee(\mathbb A)=U^\vee_+(\mathbb A)\bowtie U^\vee_-(\mathbb A)$ for
$\mathbb A=\mathbb K$. Noting that, in the commutation relation, the
monomials $C_{i,j}a^ib^j$ belong to $\mathbb O$ (respectively,
$\mathbb M$) as soon as $a$ and $b$ do so, we can use the same process
to obtain the decomposition for $\mathbb A=\mathbb O$ (respectively,
$\mathbb A=\mathbb M$). Along the way, we note that for
$\mathbb A\in\{\mathbb O,\mathbb M\}$, the element ($*$) belongs to
$U^\vee(\mathbb A)$ if and only if all the $a_\alpha$ belong
to $\mathbb A$.

We have so far proved the first item. To prove the second, we enumerate
the roots $\alpha_1$, \dots, $\alpha_n$ in $\Phi_\theta$ according to a
convex order. To simplify the notation, we write $x_i$ instead of
$x_{\alpha_i^\vee}$. Let $g\in U^\vee(\mathbb K)$. We claim that for
any $k\in\inter0n$, $g$ can be written as a product
$$x_1(b_1)\cdots x_k(b_k)\,x_{k+1}(a_{k+1})\cdots x_n(a_n)\,x_k(c_k)
\cdots x_1(c_1),$$
with $a_i$ in $\mathbb K$, $b_i\in\mathbb M$ and $c_i\in\mathbb O$.
For $k=0$, this is ($*$). Let us assume that such a factorization exists
for $k\in\inter0{n-1}$ and let us decompose $a_{k+1}=b_{k+1}+c_{k+1}$
with $(b_{k+1},c_{k+1})\in\mathbb M\times\mathbb O$. As the ordering is
convex, each coroot of the form $i\alpha_{k+1}^\vee+j\alpha_\ell^\vee$
with $i$ and $j$ positive and $\ell\geq k+2$ belongs to
$\bigl\{\alpha^\vee_{k+2},\ldots,\alpha^\vee_n\bigr\}$, so there
exists $a'_{k+2}$, \dots, $a'_n$ in $\mathbb K$ such that
$$x_{k+1}(c_{k+1})\,x_{k+2}(a_{k+2})\cdots x_n(a_n)\,
x_{k+1}(c_{k+1})^{-1}=x_{k+2}(a'_{k+2})\cdots x_n(a'_n).$$
We then obtain the desired writing
$$g=x_1(b_1)\cdots x_{k+1}(b_{k+1})\,x_{k+2}(a'_{k+2})\cdots x_n(a'_n)
\,x_{k+1}(c_{k+1})\cdots x_1(c_1)$$
for $k+1$, which establishes our claim by induction.
For $k=n$, we get a factorization $g=g'g''$ with $(g',g'')\in
U^\vee(\mathbb M)\times U^\vee(\mathbb O)$.
Moreover, the uniqueness of the writing ($*$) implies that
$U^\vee(\mathbb M)\cap U^\vee(\mathbb O)=\{1\}$.
We conclude that $U^\vee(\mathbb K)=U^\vee(\mathbb M)\bowtie
U^\vee(\mathbb O)$. Similar arguments prove the remaining decompositions.
\end{proof}

For $\mu\in P$, we denote by $R_\mu$ the orbit through $L_\mu$ under the
action of the group $U^\vee(\mathbb K)$. The stabilizer of $L_0$ under
this action is $U^\vee(\mathbb O)$, so by Lemma~\ref{le:ZappaSzep} the map
$g\mapsto g\cdot L_0$ is a bijection from $U^\vee(\mathbb M)$ onto $R_0$.

\begin{lemma}
\label{le:SmallOrbs}
For each $\mu\in P$, the intersection $S_\mu\cap R_\mu$ (respectively,
$T_\mu\cap R_\mu$) is the orbit through $L_\mu$ under the action of
the group $U^\vee_+(\mathbb K)$ (respectively,
$U^\vee_-(\mathbb K)$).
\end{lemma}

\begin{proof}
The inclusion $U^\vee_+(\mathbb K)\cdot L_\mu\subset
S_\mu\cap R_\mu$ is banal. For the reverse direction, we choose an
element $x\in S_\mu\cap R_\mu$ and write $x=g\cdot L_\mu$ with
$g\in U^\vee(\mathbb K)$. Using Lemma~\ref{le:ZappaSzep},
we decompose $g=g_+g_-$ with $g_\pm\in U^\vee_\pm(\mathbb K)$.
Since $x\in S_\mu$, we get $g_+^{-1}\cdot x\in S_\mu$. As
$S_\mu\cap T_\mu=\{L_\mu\}$, we obtain $g_-\cdot L_\mu=L_\mu$, and
therefore $x=g_+\cdot L_\mu$ belongs to $U^\vee_+(\mathbb K)\cdot L_\mu$.
\end{proof}

We define a map
$$\Omega_\mu:R_\mu\to(T_\mu\cap R_\mu)\times(S_\mu\cap R_\mu)$$
as follows. Given $x\in R_\mu$, there is a unique
$g\in U^\vee(\mathbb M)$ such that $z^{-\mu}\cdot x=g\cdot L_0$.
By Lemma~\ref{le:ZappaSzep}, we can write in a unique fashion
$g=g_+g_-=h_-h_+$ with $g_+$, $h_+$ in $U^\vee_+(\mathbb M)$ and
$g_-$, $h_-$ in $U^\vee_-(\mathbb M)$. We then set
$\Omega_\mu(x)=(z^\mu g_-\cdot L_0,z^\mu h_+\cdot L_0)$.

\begin{proposition}
\label{pr:IsomSTR}
Let $(\lambda,\mu,\nu)\in P^3$. Then the map $\Omega_\mu$ is bijective
and restricts to an isomorphism of algebraic varieties from
$S_\lambda\cap T_\nu\cap R_\mu$ onto
$(S_\lambda\cap T_\mu\cap R_\mu)\times(T_\nu\cap S_\mu\cap R_\mu)$.
\end{proposition}

\begin{proof}
We readily reduce to the case $\mu=0$.

Given $(y,z)\in(T_0\cap R_0)\times(S_0\cap R_0)$, we write
$y=g_-\cdot L_0$ and $z=h_+\cdot L_0$ with $g_-$ in
$U^\vee_-(\mathbb M)$ and $h_+$ in $U^\vee_+(\mathbb M)$, and we write
the element $g_-h_+^{-1}$ of $U^\vee(\mathbb M)$ as $g_+^{-1}h_-$ with
$(g_+^{-1},h_-)\in U^\vee_+(\mathbb M)\times U^\vee_-(\mathbb M)$.
We then have $g_+g_-=h_-h_+$, and we check that $x=g_+g_-\cdot L_0$ is
the unique element in $R_0$ such that $(y,z)=\Omega_0(x)$. The map
$\Omega_0$ is therefore bijective.

Suppose that $x\in R_0$ and set $(y,z)=\Omega_0(x)$. Then $x$ and $y$
belong to the same orbit under the action of $N^\vee_+(\mathbb K)$,
so $x\in S_\lambda$ if and only if $y\in S_\lambda$. Likewise,
$x\in T_\nu$ if and only if $z\in T_\nu$. We conclude that
$\Omega_0$ induces a bijection from $S_\lambda\cap T_\nu\cap R_0$ onto
$(S_\lambda\cap T_0\cap R_0)\times(T_\nu\cap S_0\cap R_0)$.
This bijection is in fact an isomorphism of algebraic varieties because
the definition of $\Omega_0$ and the construction of its inverse depend
only on the operations that provide the decomposition
$U^\vee(\mathbb M)=U^\vee_+(\mathbb M)
\bowtie U^\vee_-(\mathbb M)$.
\end{proof}

The subsets $R_\mu$ are locally closed and form a partition of $\Gr$,
so for any irreducible subvariety $Z$ of $\Gr$, there is a unique
weight $\mu\in P$ such that $Z\cap R_\mu$ is open and dense in $Z$.
For $(\lambda,\mu,\nu)\in P^3$, we denote by
$\mathscr Z_{\lambda,\nu}^\mu$ the subset of
$\mathscr Z_{\lambda,\nu}$ formed by the cycles $Z$
whose general point belongs to $R_\mu$. Therefore
$\mathscr Z_{\lambda,\nu}^\mu$ is the set of irreducible
components of $\overline{S_\lambda\cap T_\nu\cap R_\mu}$ of
dimension $\rho(\lambda-\nu)$.

\begin{proposition}
\label{pr:CuttingMVCycles}
The map $Z\mapsto\bigl(\overline{Z\cap T_\mu\cap R_\mu},
\overline{Z\cap S_\mu\cap R_\mu}\bigr)$ defines a bijection
$$\Xi:\mathscr Z_{\lambda,\nu}^\mu\xrightarrow\simeq
\mathscr Z_{\lambda,\mu}^\mu\times\mathscr Z_{\mu,\nu}^\mu.$$
\end{proposition}

\begin{proof}
We denote the set of irreducible components of a topological space $X$
by $\Irr(X)$. The proof of Proposition~\ref{pr:IsomSTR} implies that
$\Omega_\mu$ restricts to an isomorphism of algebraic varieties
$$\overline{S_\lambda}\cap\overline{T_\nu}\cap R_\mu\xrightarrow\simeq
(\overline{S_\lambda}\cap T_\mu\cap R_\mu)\times
(\overline{T_\nu}\cap S_\mu\cap R_\mu),$$
so induces a bijection
$$\Irr\bigl(\overline{S_\lambda}\cap\overline{T_\nu}\cap R_\mu\bigr)
\xrightarrow\simeq\Irr\bigl(\overline{S_\lambda}\cap T_\mu\cap R_\mu\bigr)
\times\Irr\bigl(\overline{T_\nu}\cap S_\mu\cap R_\mu\bigr).$$
The dimension of $\overline{S_\lambda}\cap\overline{T_\nu}\cap R_\mu$
is less than or equal to $\rho(\lambda-\nu)$, and likewise
the dimensions of $\overline{S_\lambda}\cap T_\mu\cap R_\mu$
and $\overline{T_\nu}\cap S_\mu\cap R_\mu$ are bounded by
$\rho(\lambda-\mu)$ and $\rho(\mu-\nu)$, respectively.

The map $Z\mapsto Z\cap R_\mu$ defines a bijection from
$\mathscr Z_{\lambda,\nu}^\mu$ onto the set of irreducible components
of $\overline{S_\lambda}\cap\overline{T_\nu}\cap R_\mu$ of dimension
$\rho(\lambda-\nu)$, and there are similar bijections from
$\mathscr Z_{\lambda,\mu}^\mu$ and $\mathscr Z_{\mu,\nu}^\mu$ onto the
sets of irreducible components of $\overline{S_\lambda}\cap T_\mu\cap R_\mu$
and $\overline{T_\nu}\cap S_\mu\cap R_\mu$ of dimension $\rho(\lambda-\mu)$
and $\rho(\mu-\nu)$, respectively. Therefore the above bijection produces
by restriction a bijection $\Xi:\mathscr Z_{\lambda,\nu}^\mu\xrightarrow
\simeq\mathscr Z_{\lambda,\mu}^\mu\times\mathscr Z_{\mu,\nu}^\mu$,
and to conclude the proof it remains to check that $\Xi$ is the map
described in the statement.

To simplify the notation, we shall take $\mu=0$. By construction, $\Xi$
relates a Mirković--Vilonen cycle $Z\in\mathscr Z_{\lambda,\nu}^0$ and
a pair $(Z_t,Z_s)\in\mathscr Z_{\lambda,0}^0\times\mathscr Z_{0,\nu}^0$
if and only if
$$Z\cap R_0=\bigl\{g\cdot L_0\bigm|g_-\cdot L_0\in Z_t\cap(T_0\cap R_0)
\,\text{ and }\,h_+\cdot L_0\in Z_s\cap(S_0\cap R_0)\bigr\},$$
where as before $g\in U^\vee(\mathbb M)$ is written in the
form $g_+g_-$ or $h_-h_+$ with $g_+$, $h_+$ in $U^\vee_+(\mathbb M)$
and $g_-$, $h_-$ in $U^\vee_-(\mathbb M)$.

With these notations, if $g\cdot L_0\in S_0$, then
$g\in U^\vee_+(\mathbb M)$, therefore $g=h_+$.
It follows that $Z\cap(S_0\cap R_0)\subset Z_s\cap(S_0\cap R_0)$,
whence $\overline{Z\cap S_0\cap R_0}\subset Z_s$. For the reverse
inclusion, we note that $L_0\in Z_t\cap(T_0\cap R_0)$, so (taking
$g_-=h_-=1$) $Z_s\cap(S_0\cap R_0)\subset Z\cap R_0$, and therefore
$Z_s=\overline{Z_s\cap S_0\cap R_0}\subset\overline{Z\cap S_0\cap R_0}$.
We conclude that $Z_s=\overline{Z\cap S_0\cap R_0}$. The equality
$Z_t=\overline{Z\cap T_0\cap R_0}$ is established in like manner.
\end{proof}

We have assumed thus far in this section that $\theta\in Q_{\mathbb R}^*$
is regular, but this assumption was only made to simplify the exposition.
All the constructions presented are indeed valid in general provided we
replace $\Phi_\theta$ by either
$${}'\Phi_\theta=\bigl\{\alpha\in\Phi\bigm|\theta(\alpha)<0\bigr\}\cup
\bigl\{\alpha\in\Phi_+\bigm|\theta(\alpha)=0\bigr\}$$
or
$${}''\Phi_\theta=\bigl\{\alpha\in\Phi\bigm|\theta(\alpha)<0\bigr\}\cup
\bigl\{\alpha\in\Phi_-\bigm|\theta(\alpha)=0\bigr\}.$$

\subsection{Politeness of Mirković and Vilonen's basis}
\label{ss:MVPolite}
For each dominant weight $\lambda$, the costandard module
$\nabla(\lambda)$, as defined in section~\ref{ss:GeometricSatake},
is a subspace of $\mathscr O(G)$. It contains a unique vector
$v_\lambda$ whose restriction to $N$ is the constant function equal
to $1$, and this vector has weight $\lambda$. Given a second dominant
weight $\mu$, the multiplication map $f\mapsto fv_\mu$ embeds
$\nabla(\lambda)$ as an $N$-submodule of $\nabla(\lambda+\mu)$.
The restriction to $N$, to wit the map $f\mapsto f|_N$, defines an
embedding of $N$-modules $\Psi_\lambda:\nabla(\lambda)\to\mathscr O$,
and the diagram
$$\xymatrix@!=10pt{\nabla(\lambda)\ar[rr]^(.4){\bm?v_\mu}
\ar[dr]_(.4){\Psi_\lambda}&&\nabla(\lambda+\mu)
\ar[dl]^(.4){\Psi_{\lambda+\mu}}\\&\mathscr O&}$$
commutes. In this fashion $\mathscr O$ is identified with the direct
limit of the spaces $\nabla(\lambda)$.

As we saw, the Geometric Satake Equivalence provides an identification
$$\nabla(\lambda)\cong\bigoplus_{\nu\in P}
\text H_{2\rho(\lambda-\nu)}^\BM\bigl(\Gr^\lambda\cap T_\nu,k\bigr)$$
and the fundamental classes of the irreducible components of
$\Gr^\lambda\cap T_\nu$, for all $\nu\in P$, form a basis of
$\nabla(\lambda)$. Now let $\nu\in Q_-$ and let $Z\in\mathscr Z_{0,\nu}$.
If $\lambda$ is a dominant enough weight, then $Z$ is contained
in $z^{-\lambda}\,\overline\Gr^\lambda$. In this case, $z^\lambda Z$
is the closure of an irreducible component $Y$ of
$\Gr^\lambda\cap T_{\lambda+\nu}$ (\cite{Anderson}, Proposition~3);
furthermore, the image by $\Psi_\lambda:\nabla(\lambda)\to\mathscr O$
of the fundamental class of $Y$ depends only on $Z$ and not on
$\lambda$ (\cite{BaumannKamnitzerKnutson}, Proposition~6.1).
We denote this image by $b_Z$. For each $\nu\in Q_-$, the elements
$b_Z$ with $Z\in\mathscr Z_{0,\nu}$ form a basis of $\mathscr O_\nu$.
The Mirković--Vilonen basis of $\mathscr O$ is the family of all
these functions $b_Z$.

Let us resume the discussion in the last paragraph of
section~\ref{ss:GeometricSatake}, with the projective embedding
$\Upsilon:\Gr\to\mathbb P(V)$ derived from the choice of a quadratic
form $q$. Let $\mathfrak H$ be the direct sum of all weight subspaces
of $V$ save for the highest weight (with respect to the Cartan
subalgebra of the semi-direct product
$\widetilde{\mathfrak g^\vee}\rtimes\mathbb C\frac d{dz}$);
this is a hyperplane of $V$. We denote by $D$ the divisor on $\Gr$ cut
by $\mathbb P(\mathfrak H)$, and for $\nu\in P$, we set $D_\nu=z^\nu D$.
By~\cite{MirkovicVilonen}, Proposition~3.1, we then have
$$D_\nu\cap\overline{T_\nu}=\bigcup_{i\in I}\overline{T_{\nu+\alpha_i}}.$$
Given two Mirković--Vilonen cycles $Z$ and $Z'$ of weights $(\lambda,\nu)$
and $(\lambda,\nu+\alpha_i)$, respectively, we denote by
$i(Z',D_\nu\cdot Z)$ the multiplicity of $Z'$ in the intersection
product $D_\nu\cdot Z$.

The basis elements $b_Z$ can now be explicitly described in the
following way. If $Z$ is a Mirković--Vilonen cycle of weight $(0,\nu)$
and $\mathbf i=i_1\cdots i_n$ is a word of weight $-\nu$, we denote by
$\mathscr C_{\mathbf i,Z}$ the set of all chains
$\{L_0\}=Y_0\subset Y_1\subset\cdots\subset Y_n=Z$ of Mirković--Vilonen
cycles such that $Y_j\in\mathscr Z_{0,\nu_j}$ for each $j\in\inter0n$,
where $\nu_j=-(\alpha_{i_1}+\cdots+\alpha_{i_j})$.
By~\cite{BaumannKamnitzerKnutson}, Theorem~5.4, we then have
$$\langle e_{\mathbf i},b_Z\rangle=\frac1{q(\alpha_{i_1})\cdots
q(\alpha_{i_n})}\sum_{Y_\bullet\in\mathscr C_{\mathbf i,Z}}
i(Y_0,D_{\nu_1}\cdot Y_1)\cdots i(Y_{n-1},D_{\nu_n}\cdot Y_n).$$

We will need the following technical lemmas. The second one confirms
in part Anderson's description of the coproduct of $\mathscr O$ in
the Mirković--Vilonen basis (\cite{Anderson}, section~10). In both
statements, we assume that a stability parameter $\theta$ has been
fixed in $Q_{\mathbb R}^*$ and we adopt the notation of
section~\ref{ss:CuttingMVCycles}, with $\Phi_\theta$ replaced by either
${}'\Phi_\theta$ or ${}''\Phi_\theta$ if $\theta$ is not regular.

\begin{lemma}
\label{le:Sublemma}
Let $(\lambda,\mu,\nu)\in P^3$, let $i\in I$, and
let $(Z,Z')\in\mathscr Z_{\lambda,\nu}^\mu\times
\mathscr Z_{\lambda,\nu+\alpha_i}^\mu$ be such that $Z'\subset Z$.
Write $\Xi(Z)=(Z_t,Z_s)$ and $\Xi(Z')=(Z'_t,Z'_s)$. Then $Z_t=Z'_t$
and $i(Z'_s,D_\nu\cdot Z_s)=i(Z',D_\nu\cdot Z)$.
\end{lemma}

\begin{proof}
We observe that $Z'_t=\overline{Z'\cap T_\mu\cap R_\mu}\subset
\overline{Z\cap T_\mu\cap R_\mu}=Z_t$. As both $Z_t$ and $Z'_t$ are
Mirković--Vilonen cycles of weight $(\lambda,\mu)$, they have the
same dimension, and therefore the inclusion is an equality.

We set $\dot Z=Z\cap R_\mu$, $\dot Z_t=Z\cap T_\mu\cap R_\mu$,
$\dot Z_s=Z\cap S_\mu\cap R_\mu$, $\dot Z'=Z'\cap R_\mu$
and $\dot Z'_s=Z'\cap S_\mu\cap R_\mu$. These are affine varieties.
The map $\Omega_\mu$ restricts to an isomorphism of algebraic
varieties $\widetilde\Omega_\mu:\dot Z\to\dot Z_t\times\dot Z_s$.
Let $\text{pr}_2:\dot Z_t\times\dot Z_s\to\dot Z_s$ be the second
projection. We have
$\widetilde\Omega_\mu(\dot Z')=\dot Z_t\times\dot Z'_s$, that is,
$\dot Z'=(\text{pr}_2\circ\widetilde\Omega_\mu)^{-1}(\dot Z'_s)$.

The Lie algebra $\widetilde{\mathfrak g^\vee}$ integrates to a
central extension $E(G^\vee(\mathbb K))$ of the loop group
$G^\vee(\mathbb K)$, so that $V$ is an ordinary representation of
$E(G^\vee(\mathbb K))$. The cocycle on $G^\vee(\mathbb K)$ that
defines this central extension $E(G^\vee(\mathbb K))$ is trivial
when restricted to the subgroups $N^\vee_\pm(\mathbb K)$ or
$U^\vee(\mathbb K)$. Therefore these subgroups lift (non-canonically)
to the central extension, in other words their action on $\Gr$ lifts
to an action on $V$.

Let $\sigma$ and $\tau$ be linear forms on $V$ with kernels
$z^\nu\mathfrak H$ and $z^\mu\mathfrak H$, respectively. Then $\tau$
is nonzero on the line $\Upsilon(L_\mu)$ and it vanishes on all
the other weight subspaces of $V$, so it takes a constant nonzero
value on the $U^\vee(\mathbb K)$-orbit of any nonzero vector $v_1$
in $\Upsilon(L_\mu)$. In particular $\tau$, regarded as a section
of the line bundle $\mathscr L$, does not vanish on $R_\mu$.
If we also regard $\sigma$ as a section of~$\mathscr L$, then
$f=\sigma/\tau$ is a rational function on $\Gr$ which is regular on
$R_\mu$. This function $f$ is the equation of the divisor $D_\nu$ on
$R_\mu$.

We claim that if $x$ is in $\dot Z$ and if
$y=\text{pr}_2\circ\widetilde\Omega_\mu(x)$, then $f(x)=f(y)$.
To show this, we write $x=g\cdot L_\mu$ with
$g\in z^\mu\;U^\vee(\mathbb M)\,z^{-\mu}$, we decompose $g=h_-h_+$
with $h_\pm\in z^\mu\;U^\vee_\pm(\mathbb M)\,z^{-\mu}$, and we note
that $y=h_+\cdot L_\mu$ and $x=h_-\cdot y$. The closure
$\overline{T_\nu}$ contains $Z$ and is the disjoint union of the
semi-infinite orbit $T_\nu$ and its boundary $D\cap\overline{T_\nu}$,
which are both $N^\vee_-(\mathbb K)$-invariant. The points $x$ and $y$
are therefore both either in $D\cap\overline{T_\nu}$ or in $T_\nu$.
In the first case, we have $f(x)=f(y)=0$. In the second case, we note
that the vector $v_3=h_+\cdot v_1$ belongs to the line $\Upsilon(y)$.
As $\tau$ assumes a constant value on the orbit of $v_1$ under the
action of $U^\vee(\mathbb K)$, we have $\tau(h_-\cdot v_3)=\tau(v_3)$.
As $y$ is assumed to be in $T_\nu$, there exists
$g'\in N^\vee_-(\mathbb K)$ such that $y=g'\cdot L_\nu$,
so $v_2=g'^{-1}\cdot v_3$ is in $\Upsilon(L_\nu)$. We can now repeat
our argument: $\sigma$ is nonzero on the line $\Upsilon(L_\nu)$
and it vanishes on all the other weight subspaces of $V$, so it takes a
constant nonzero value on the $N^\vee_-(\mathbb K)$-orbit of $v_2$.
We deduce that $\sigma(h_-\cdot v_3)=\sigma(v_3)$, and therefore
$f(x)=f(y)$ in this case as well. Our claim is proved.

In conclusion, the equation of the divisor $D_\nu$ and the equation
of the subvariety $\dot Z'$ inside $\dot Z$ both factorize through
$\text{pr}_2\circ\widetilde\Omega_\mu$.
This implies that $i(Z',D_\nu\cdot Z)=i(Z'_s,D_\nu\cdot Z_s)$.
\end{proof}

\begin{lemma}
\label{le:CutCoprod}
Let $(\mu,\nu)\in P\times Q_-$ and let $Z\in\mathscr Z_{0,\nu}^\mu$.
\begin{enumerate}
\item
We have $L_\mu\in Z$, and if an element $L_\eta$ belongs to $Z$, then
$\mu-\eta$ is a nonnegative linear combination of roots in $\Phi_\theta$.
\item
Cut $Z$ in two Mirković--Vilonen cycles
$Z_t=\overline{Z\cap T_\mu\cap R_\mu}$ and
$Z_s=\overline{Z\cap S_\mu\cap R_\mu}$.
Then the component of weight $(\mu,\nu-\mu)$ of $\Delta(b_Z)$
is equal to $b_{Z_t}\otimes b_{z^{-\mu}Z_s}$.
\end{enumerate}
\end{lemma}

\begin{proof}
The first item is standard: $L_\mu$ is in $Z$ because $Z$ meets $R_\mu$
and is closed and invariant under the action of $T^\vee(\mathbb C)$,
and the second point follows from the inclusion $Z\subset\overline{R_\mu}$.

Since $\mathscr Z_{0,\nu}^\mu$ is non-empty, both $\mu$ and $\nu-\mu$
are in $Q_-$. Let $\mathbf i=i_1\cdots i_m$ and
$\mathbf j=i_{m+1}\cdots i_n$ be words in the alphabet $I$ of weights
$-\mu$ and $-\nu+\mu$, respectively. We set
$\nu_j=-(\alpha_{i_1}+\cdots+\alpha_{i_j})$ for each $j\in\inter0n$.
Let $Y_\bullet$ be an element in $\mathscr C_{\mathbf{ij},Z}$. For each
$j\in\inter0n$, we have $Y_j\subset Z\subset\overline{R_\mu}$, so
$Y_j\cap R_\mu$ is open in $Y_j$. We observe that $L_\mu\in Y_m$, so
$Y_j\cap R_\mu$ is non-empty if $j\geq m$. Therefore for all
$j\in\inter mn$, the general point of $Y_j$ belongs to $R_\mu$. Since
both $Y_m\cap T_\mu$ and $Y_m\cap R_\mu$ are open dense in $Y_m$, we have
$$Y_m=\overline{Y_m\cap T_\mu\cap R_\mu}\subset\overline{Z\cap T_\mu
\cap R_\mu}=Z_t.$$
This implies that $Y_m=Z_t$, because both $Y_m$ and $Z_t$ are
Mirković--Vilonen cycles of weight $(0,\mu)$, and we deduce that
$\overline{Y_j\cap T_\mu\cap R_\mu}=Z_t$ for all $j\in\inter mn$.
In the notation of Proposition~\ref{pr:CuttingMVCycles}, we can then
write $\Xi(Y_j)=(Z_t,Y'_j)$, where
$Y'_j=\overline{Y_j\cap S_\mu\cap R_\mu}$ is in $\mathscr Z_{\mu,\nu_j}$.
Plainly then, $(Y_0,\ldots,Y_m)\in\mathscr C_{\mathbf i,Z_t}$ and
$(z^{-\mu}Y'_m,\ldots,z^{-\mu}Y'_n)\in\mathscr C_{\mathbf j,z^{-\mu}Z_s}$.
We can thus define a map $\mathscr C_{\mathbf{ij},Z}\to
\mathscr C_{\mathbf i,Z_t}\times\mathscr C_{\mathbf j,z^{-\mu}Z_s}$,
which is patently bijective.

From Lemma~\ref{le:Sublemma}, we deduce that for all $j\in\inter{m+1}n$,
$$i(Y_{j-1},D_{\nu_j}\cdot Y_j)=i(Y'_{j-1},D_{\nu_j}\cdot Y'_j)
=i(z^{-\mu}Y'_{j-1},D_{\nu_j-\mu}\cdot z^{-\mu}Y'_j).$$
A straightforward calculation then shows that
$\langle e_{\mathbf{ij}},b_Z\rangle=\langle e_{\mathbf i},
b_{Z_t}\rangle\langle e_{\mathbf j},b_{z^{-\mu}Z_s}\rangle$.
The truthfulness of this equality for all words $\mathbf i$ and
$\mathbf j$ of weights $-\mu$ and $-\nu+\mu$ establishes that the
component of weight $(\mu,\nu-\mu)$ of $\Delta(b_Z)$ is equal to
$b_{Z_t}\otimes b_{z^{-\mu}Z_s}$.
\end{proof}

The group $G^\vee(\mathbb K)$ and its subgroup $T^\vee(\mathbb C)$ act
on the affine Grassmannian $\Gr$. We define the moment polytope of a
closed irreducible $T^\vee(\mathbb C)$-invariant subvariety $X$ of
$\Gr$ as the convex hull in $P\otimes_{\mathbb Z}\mathbb R$ of the set
$\{\mu\in P\mid L_\mu\in X\}$. We refer to \cite{Anderson}, section~6
for an explanation of the terminology `moment polytope'. The moment
polytope of a closed irreducible $T^\vee(\mathbb C)$-invariant
subvariety is always a GGMS polytope (see \textit{loc.~cit.};
see also Lemma~2.3 in~\cite{Kamnitzer10}).

\begin{proposition}
\label{pr:MomentPolytope}
The moment polytope of a Mirković--Vilonen cycle $Z$ is equal to
$\underline\Pol(b_Z)$.
\end{proposition}

\begin{proof}
Let $\nu$ be the weight of $b_Z$, so that $Z\in\mathscr Z_{0,\nu}$.
Let us pick $\mu\in\underline L(b_Z)$. The component of weight
$(\mu,\nu-\mu)$ of $\Delta(b_Z)$ is nonzero, so there exist words
$\mathbf i$ and $\mathbf j$ of weights $-\mu$ and $-\nu+\mu$,
respectively, such that $\langle e_{\mathbf{ij}},b_Z\rangle\neq0$.
Therefore $\mathscr C_{\mathbf{ij},Z}\neq\varnothing$, so there exists a
Mirković--Vilonen cycle $Z'\in\mathscr Z_{0,\mu}$ which is contained in
$Z$. Then $L_\mu\in Z$ and $\mu$ belongs to the moment polytope of $Z$.
We conclude that $\underline L(b_Z)$, and therefore its convex hull
$\underline\Pol(b_Z)$, is contained in the moment polytope of $Z$.

To prove the reverse inclusion, we choose a vertex $\mu$ of the moment
polytope of $Z$. Let $\theta$ be a regular element in $Q_{\mathbb R}^*$
such that $-\theta$ lies in the normal cone at $\mu$ of this polytope.
Adopting the notation of section~\ref{ss:CuttingMVCycles}, we claim
that the general point of $Z$ belongs to the semi-infinite orbit $R_\mu$.

In fact, denote by $\lambda$ the weight such that $Z\cap R_\lambda$ is
open dense in $Z$. By Lemma~\ref{le:CutCoprod}, as $L_\mu$ belongs to
$Z$, the difference $\lambda-\mu$ is a nonnegative linear combination
of roots $\alpha$ such that $\theta(\alpha)<0$. It follows that
$\theta(\lambda-\mu)\leq0$, with equality only if $\lambda=\mu$. On the
other hand, $\lambda$ belongs to the moment polytope of $Z$, and our choice
of $\theta$ implies that $\theta(\lambda-\mu)\geq0$. We then conclude that
$\lambda=\mu$, as announced.

Returning to the main proof, we deduce from Lemma~\ref{le:CutCoprod}
that the component of weight $(\mu,\nu-\mu)$ of $\Delta(b_Z)$ is
equal to $b_{Z_t}\otimes b_{z^{-\mu}Z_s}$ with
$(Z_t,Z_s)\in\mathscr Z_{0,\mu}\times\mathscr Z_{\mu,\nu}$. This is
not zero, so $\mu\in\underline L(b_Z)$. We conclude that all the
vertices of the moment polytope of $Z$ belong to $\underline\Pol(b_Z)$.
\end{proof}

\begin{theorem}
\label{th:MVPolite}
The Mirković--Vilonen basis of $\mathscr O$ is polite.
\end{theorem}

\begin{proof}
By construction, the Mirković--Vilonen basis satisfies the
condition~(i) in the definition of polite bases. The condition~(ii)
follows from Proposition~5.5 in~\cite{BaumannKamnitzerKnutson}.

We now show that the Mirković--Vilonen basis satisfies condition~(iii).
Let $\theta\in Q_{\mathbb R}^*$, let $(\alpha,\gamma)\in\Gamma_\theta^2$,
and let $f$ be a homogenous element in ${}^\theta F'_\alpha\mathscr O$
of degree $\gamma$. Let us expand $f$ on the Mirković--Vilonen basis
and write $f=\sum_{i=1}^\ell c_ib_{Z_i}$, where each scalar
$c_i$~is~nonzero and where $Z_i$ is in $\mathscr Z_{0,\nu_i}$ for a
certain weight $\nu_i$. We adopt the notation of
section~\ref{ss:CuttingMVCycles} with $\Phi_\theta$ replaced by
${}'\Phi_\theta$. We denote by $\mu_i$ the weight such that the
general point of $Z_i$ belongs to the semi-infinite orbit~$R_{\mu_i}$
and we cut $Z_i$ at $L_{\mu_i}$, producing cycles $(Z_i)_t$ and $(Z_i)_s$
in $\mathscr Z_{0,\mu_i}$ and $\mathscr Z_{\mu_i,\nu_i}$, respectively.

By the proof of Proposition~\ref{pr:MomentPolytope}, any element
$\eta\in\underline L(b_{Z_i})$ verifies $L_\eta\in Z_i$. It then
follows from Lemma~\ref{le:CutCoprod} that $\mu_i-\eta$ is a nonnegative
linear combination of roots $\alpha$ such that either $\theta(\alpha)<0$,
or $\theta(\alpha)=0$ and $\height\alpha>0$. Therefore
$J_\theta(\eta)\leq'J_\theta(\mu_i)$, with equality only if $\eta=\mu_i$.
We conclude that $J_\theta(\mu_i)$ is the largest element of
$L_\theta(b_{Z_i})$ with respect to the order $\leq'$.

Let $\beta$ be the largest element among the degrees $J_\theta(\mu_i)$
with respect to the total order $\leq'$ and let $I$ be the set of all
indices $i$ such that $J_\theta(\mu_i)=\beta$.
Then the homogeneous component of degree
$(\beta,\gamma-\beta)$ of $\Delta(f)$ is
$\sum_{i\in I}c_i\,b_{(Z_i)_t}\otimes b_{z^{-\mu_i}(Z_i)_s}$. Since the
pairs $((Z_i)_t,(Z_i)_s)$ are all different, this sum cannot be zero,
so $\beta\in L(f)$ and therefore $\beta\leq'\alpha$. It follows that
for each $i\in\inter1\ell$, the set $L_\theta(b_{Z_i})$ is
contained in $\alpha+\Pi'$ and the function $b_{Z_i}$
belongs to ${}^\theta F'_\alpha\mathscr O$. We conclude that
${}^\theta F'_\alpha\mathscr O$ is spanned by its intersection with
the Mirković--Vilonen~basis. The same arguments imply the
compatibility of the Mirković--Vilonen basis with the isomorphism
$\overline\Delta'$.

Condition~(iv) in the definition of polite bases can be checked
similarly, replacing instead $\Phi_\theta$ by
${}''\Phi_\theta$ in section~\ref{ss:CuttingMVCycles}.
\end{proof}

\begin{other}{Example}
Let $W$ be the Weyl group of $G$, endowed with the Bruhat order.
Let $\lambda$ be a dominant weight and let $w$ be an element in $W$.
The costandard module $\nabla(\lambda)$ contains a vector
$v_{w\lambda}$ of weight $w\lambda$, unique up to multiplication by
a scalar. Let $\nabla(\lambda)_w$ be the $N$-submodule of
$\nabla(\lambda)$ generated by $v_{w\lambda}$
and let $f_{w,\lambda}$ be the image of $v_{w\lambda}$ by the map
$\Psi_\lambda:\nabla(\lambda)\to\mathscr O$. It is well-known that
$f_{w,\lambda}$ belongs to all the perfect bases of $\mathscr O$
(provided that $v_{w\lambda}$ is suitably normalized), therefore to
all the polite bases. In particular, $\underline\Pol(f_{w,\lambda})$ is
a Mirković--Vilonen polytope. Now, since $\Psi_\lambda$ is an embedding
of $N$-modules, the polytope $\lambda+\underline\Pol(f_{w,\lambda})$
is the convex hull of the set of weights of $\nabla(\lambda)_w$.
On the other hand, it follows from the Demazure character formula that
this polytope is the convex hull of $\{x\lambda\mid x\in W,\;x\leq w\}$.
We recover in this way a result of Naito and Sagaki~(\cite{NaitoSagaki},
Theorem~4.1.5 and~Remark~4.1.6).
\end{other}

Pierre Baumann,
Institut de Recherche Mathématique Avancée,
Université de Strasbourg et CNRS UMR 7501,
7 rue René Descartes,
67000 Strasbourg,
France.\\
\texttt{baumann@math.unistra.fr}

\end{document}